\documentclass[11pt,a4paper]{amsart}

\usepackage[T1]{fontenc}
\usepackage[utf8]{inputenc}
\usepackage[english]{babel}
\usepackage{amsmath,amssymb,amsthm}
\usepackage[margin=1.1in]{geometry}
\usepackage{hyperref}
\hypersetup{colorlinks=true,linkcolor=blue,citecolor=blue}

\theoremstyle{plain}
\newtheorem{theorem}{Theorem}[section]
\newtheorem{proposition}[theorem]{Proposition}
\newtheorem{lemma}[theorem]{Lemma}
\newtheorem{corollary}[theorem]{Corollary}

\newtheorem*{thmA}{Theorem A}
\newtheorem*{thmB}{Theorem B}

\theoremstyle{definition}
\newtheorem{definition}[theorem]{Definition}
\newtheorem{remark}[theorem]{Remark}
\newtheorem{problem}[theorem]{Problem}

\newcommand{\M}{\mathbb{M}}
\newcommand{\R}{\mathbb{R}}
\newcommand{\Sph}{\mathbb{S}}
\newcommand{\Hyp}{\mathbb{H}}
\newcommand{\ct}{\operatorname{ct}_\kappa}
\newcommand{\sn}{\operatorname{sn}_\kappa}
\newcommand{\Hess}{\operatorname{Hess}}
\newcommand{\Ric}{\operatorname{Ric}}
\newcommand{\tr}{\operatorname{tr}}
\newcommand{\nul}{\operatorname{nul}}
\newcommand{\dist}{\operatorname{dist}}
\newcommand{\bconn}{\overline{\nabla}}
\newcommand{\Ker}{\mathcal{K}}

\title[Degenerate free boundary minimal annuli beyond the hemisphere]
{A degenerate free boundary minimal annulus and non-uniqueness in spherical caps beyond the hemisphere}

\author{Alexander Pigazzini}

\date{}

\subjclass[2020]{Primary 53A10; Secondary 53C42, 58J50}

\keywords{Free boundary minimal surfaces, spherical caps, minimal annuli, degenerate annulus, non-uniqueness, Jacobi fields, Robin boundary condition, nullity, capillary hypersurface, contact angle.}

\begin{document}

\begin{abstract}
Let $\{\Sigma_a\}_{a\in(0,1/2)}$ be de Oliveira's family of embedded free boundary minimal annuli of revolution contained in geodesic balls $B(R(a))\subset\Sph^3$ of radius $R(a)>\tfrac\pi2$. We prove that the radius map $R$ is real-analytic and tends to $\tfrac\pi2$ at both ends of the parameter range, and therefore \emph{folds}: it is non-monotone, attains an interior maximum $R_*>\tfrac\pi2$, and is not injective. Two consequences follow. First, every geodesic ball of radius strictly between $\tfrac\pi2$ and $R_*$ contains at least two, and at most finitely many, mutually non-congruent embedded free boundary minimal annuli of the family. Second, at every critical point of $R$ the annulus is \emph{degenerate}: it carries a rotationally invariant, reflection-even Jacobi--Robin field which is \emph{not} induced by any Killing field of $\Sph^3$ preserving the ball, and its nullity is at least three. The critical set is a nonempty discrete subset of $(0,\tfrac12)$, so these annuli are isolated in a family whose remaining members have vanishing nullity in that sector; those sitting at a maximizer of $R$, hence in the largest cap $B(R_*)$ the family reaches, we call \emph{degenerate annuli of maximal cap radius}. Along the family the critical points of the area coincide with those of $R$, so that a degenerate annulus of maximal cap radius is also a local maximizer of the area within the family. In particular the hypothesis that all Jacobi fields of an embedded free boundary minimal annulus in a spherical cap are Killing-induced, which underlies the continuity approach to uniqueness of Naff--Zhu, does not survive the passage beyond $R=\tfrac\pi2$. The degeneration is detected by the exact identity $\dim\Ker_0^{\mathrm{ev}}(\Sigma_a)=\mathbf{1}_{\{r'=0\}}(a)$, itself a consequence of a boundary relation valid in every dimension and with no symmetry assumption: for a family of free boundary minimal hypersurfaces in concentric geodesic balls $B(r(a))$, the normal variation field satisfies a \emph{Robin defect identity}: $\partial_\eta\varphi_a-\ct(r(a))\varphi_a=r'(a)A_a(\eta,\eta)$, which requires of the ambient only that the barrier family be umbilic, and which extends to capillary boundary conditions at constant contact angle.
\end{abstract}

\maketitle

\section{Introduction}\label{sec:intro}

Free boundary minimal surfaces in geodesic balls of space forms have been the object of intense recent activity. In the Euclidean unit ball the theory is classical since Fraser--Schoen \cite{FS16}; in geodesic balls of $\Sph^3$ and of $\Hyp^3$ the study was initiated by Lima--Menezes \cite{LM23} and Medvedev \cite{Med23}, and has grown rapidly: uniqueness results under symmetry \cite{Lim25}, non-rotational examples of Cerezo \cite{Cer25} and Cerezo--Fern\'andez--Mira \cite{CFM25}, rotational families in spherical caps due to de Oliveira \cite{dO24}, and the systematic programme for caps of Naff--Zhu \cite{NZ1} and Zhu \cite{NZ2}; see also \cite{NZhalf,FS15,FL14,KM} for related background.

Two features separate the spherical setting from the Euclidean one. First, there is no scaling: a one-parameter family of free boundary minimal surfaces genuinely explores balls of different radii, and the radius becomes a parameter carrying spectral information. Second, balls larger than a hemisphere are available, and very little is known there. De Oliveira \cite[Prop. 3.4]{dO24} produced a one-parameter family $\{\Sigma_a\}_{a\in(0,1/2)}$ of embedded free boundary minimal annuli of revolution contained in geodesic balls $B(R(a))\subset\Sph^3$ with $R(a)>\tfrac\pi2$; the behavior of the radius map $a\mapsto R(a)$ was left open, and its bijectivity on the complementary sub-hemispherical branch is explicitly raised as an open problem in \cite[Rem. 3.5]{dO24} and revisited in \cite[\S 6.1]{NZ1}.

Our main result determines the qualitative behavior of $R$ on the super-hemispherical branch and draws the spectral consequences.

\begin{thmA}
Let $\{\Sigma_a\}_{a\in(0,1/2)}$ be the family of embedded free boundary minimal annuli of revolution in $B(R(a))\subset\Sph^3$ of \textup{\cite[Prop. 3.3 and Prop. 3.4]{dO24}}, described in \S\textup{\ref{sec:caps}} below. Then:
\begin{enumerate}
\item[\textup{(i)}] the radius map $R$ is real-analytic on $(0,\tfrac12)$, satisfies $R>\tfrac\pi2$ there, and
\[
\lim_{a\to 0^+}R(a)=\lim_{a\to (1/2)^-}R(a)=\frac\pi2 ;
\]
in particular $R$ is non-monotone, it attains an interior maximum, and it is not injective;
\item[\textup{(ii)}] for every $\rho\in\bigl(\tfrac\pi2,\max R\bigr)$ the geodesic ball $B(\rho)\subset\Sph^3$ contains at least two, and at most finitely many, mutually non-congruent embedded free boundary minimal annuli of the family;
\item[\textup{(iii)}] at every critical point $a_*$ of $R$, and in particular at an interior maximum, the annulus $\Sigma_{a_*}$ is degenerate modulo the ambient isometries preserving the ball (Definition \ref{def:degannulus}), with $\nul(\Sigma_{a_*})\geq 3$: the Jacobi--Robin kernel of $\Sigma_{a_*}$ contains the nonzero, rotationally invariant, reflection-even field $\varphi_{a_*}$, and $\varphi_{a_*}$ is \emph{not} induced by any Killing field of $\Sph^3$ whose flow preserves $B(R(a_*))$;
\item[\textup{(iv)}] the critical set of $R$ is discrete, and off it the rotationally invariant reflection-even Robin nullity vanishes identically.
\end{enumerate}
\end{thmA}

We stress the reading of (iii) and (iv) together. The critical set of $R$ is discrete, so the degenerate annuli produced by Theorem A are \emph{isolated}: what the theorem exhibits is not a family of degenerate annuli, but a nonempty discrete set of them inside a family whose remaining members have vanishing nullity in that sector. An annulus sitting at an interior maximum of $R$ realizes the largest cap radius attained by the family; we call it \emph{a degenerate annulus of maximal cap radius} (Definition \ref{def:degannulus}). Whether the maximizer is unique is Problem \ref{prob:caps}; no statement below depends on it.

Item (iii) has a consequence for the continuity approach to uniqueness of free boundary minimal annuli in caps developed in \cite[\S 6]{NZ1}. That approach requires knowing that all Jacobi fields of an embedded free boundary minimal annulus in a cap are induced by ambient Killing fields, a property raised as \cite[Question 6.6]{NZ1}. Theorem A(iii) shows that the de Oliveira family contains an embedded free boundary minimal annulus, in a ball of radius larger than $\tfrac\pi2$, for which this fails. We stress the scope of this statement: \cite[Question 6.6]{NZ1} is posed inside a section whose standing framework is $R\in[0,\tfrac\pi2]$, and our example lives strictly above the hemisphere; see Remark \ref{rem:scope}. Thus we do not contradict the question as posed, but we show that its hypothesis does not survive the passage beyond $R=\tfrac\pi2$. This is consistent with, and makes rigorous, the numerical picture recorded in \cite[Rem. 1.3]{NZ1}, where the radius of the rotational family is observed to reach a maximum $\bar R>\tfrac\pi2$ and then return to $\tfrac\pi2$, and where uniqueness is accordingly expected to fail for $R\in(\tfrac\pi2,\bar R)$; what is new here is the proof, not the phenomenon.

The mechanism behind (iii) and (iv) is elementary and general, and we isolate it because it applies well beyond the case at hand. Let $\{\Sigma_a\}$ be a smooth family of free boundary minimal hypersurfaces in concentric geodesic balls $B(r(a))$ of a space form, and let $\varphi_a$ denote the normal component of the variation field. Since every member is minimal, $\varphi_a$ solves the Jacobi equation; but the boundary condition of the fixed-radius second variation problem is of Robin type, and $\varphi_a$ has no reason to satisfy it, because the family does not preserve the ball. The failure is completely explicit.

\begin{thmB}
Let $\{\Sigma_a\}_{a\in I}$ be a smooth one-parameter family of free boundary minimal hypersurfaces in concentric geodesic balls $B(r(a))$ of a space form $\M^{n+1}_\kappa$, in the sense of Definition \textup{\ref{def:family}}. Then, for every $a\in I$, on $\partial\Sigma$,
\[
\partial_\eta\varphi_a-\ct(r(a))\,\varphi_a\;=\;r'(a)\,A_a(\eta,\eta),
\]
where $\eta$ is the outward conormal, $A_a$ the second fundamental form and $\ct=\coth,\ 1/r,\ \cot$ in the hyperbolic, Euclidean and spherical case. No symmetry is assumed, in any dimension.
\end{thmB}

The absence of symmetry is not decorative. The computation produces a mixed term $A_a(w,\eta)$ with $w$ tangent to $\partial\Sigma$, which vanishes identically because a free boundary on a totally umbilic barrier is a line of curvature (Lemma \ref{lem:curvline}); this is a form of Joachimsthal's classical theorem, and its capillary counterpart is \cite[Lem. 2.5]{NZ1}. Theorem B degenerates to the classical Euclidean support function relation along the dilation family and to the zero-defect property of Killing-induced Jacobi fields when $r'\equiv 0$ (Corollaries \ref{cor:killing} and \ref{cor:euclid}).

Paired with the Jacobi--Robin kernel, Theorem B yields at once the dictionary
\[
\varphi_a\ \text{is a Jacobi--Robin field}\qquad\Longleftrightarrow\qquad r'(a)=0
\]
under the sole nondegeneracy assumption $A_a(\eta,\eta)\not\equiv 0$ (Theorem \ref{thm:dictionary}), and, for reflection-symmetric rotational annuli in three-dimensional balls, the exact spectral identity
\[
\dim\Ker_0^{\mathrm{ev}}(\Sigma_a)=\mathbf{1}_{\{r'=0\}}(a)
\]
of Theorem \ref{thm:exactnul}, where $\Ker_0^{\mathrm{ev}}$ denotes the space of rotationally invariant, reflection-even Jacobi--Robin fields.

\medskip
\noindent\textbf{A dichotomy.} The exact identity cuts both ways, and it is worth recording the two conclusions it produces side by side, since they are opposite and both analytic. In $\Sph^3$ above the hemisphere the radius map folds, and Theorem A(iii) converts the fold into a genuine degeneration. In $\Hyp^3$, on the rotational family of Mori \cite{Mori81} and do Carmo--Dajczer \cite{dCD83}, the radius is instead strictly increasing on a right neighbourhood of the parameter at which the neck closes, by the analytic work of \cite{P1}, and Corollary \ref{cor:mori-local} below then gives that the family carries \emph{no} parametric degeneration there and that its rotationally invariant even Robin nullity vanishes identically on that regime. Whether the hyperbolic radius is globally monotone is Problem \ref{prob:mori}, raised in \cite[Rem.~4.3]{dO-thesis} and independently in \cite[Question 6.3]{P2}.

\medskip
\noindent\textbf{Relation to previous work.} Three families of boundary identities are close relatives of Theorem B. In the Euclidean ball, the support function $u=\langle X,\nu\rangle$ of a free boundary minimal hypersurface satisfies $\partial_\eta u=A(\eta,\eta)$ on $\partial\Sigma$, a relation underlying index computations and gap theorems \cite{FS16,AN16,Dev19,Tra20,SZ19}. Normal components of ambient Killing fields preserving the ball satisfy the Robin condition exactly; instances appear in Li--Xiong \cite[eq. (26)]{LX18}. Boundary relations at general contact angle, for quantities of a different nature, are established in \cite[Lem. 2.9]{NZ1}. Finally, the pairing of a Jacobi field with nonzero Robin defect against a putative Jacobi--Robin field, which we record in general form in Proposition \ref{prop:pairing}, is the mechanism used in \cite{P1} to exclude kernel elements in the rotationally invariant sector for the Mori family in $\Hyp^3$, where the defect identity is established for that family in \cite[Lem. 8.3]{P1}, in the form $\partial_\eta\phi_a-\coth r(a)\phi_a=-r'(a)\,\mathrm{II}(\eta,\eta)$; the present paper extracts the mechanism from that special case. To the author's knowledge the defect identity for the parametric field of a variable-radius family has not been recorded in the generality of Theorem B.
A different spectral characterisation of free boundary minimal submanifolds of geodesic balls in space forms appears in \cite[Prop.~2.6, Thms.~2.7--2.8]{dO-thesis}, obtained in the author's Candidacy Exam Report of September 2021, independently in \cite[Prop.~1]{LM23} for $r<\pi/2$, and in \cite[Prop.~2.1]{Med23} in the hyperbolic case. There the ambient coordinate functions satisfy \emph{homogeneous} Robin conditions with the same coefficients $1/r$, $\cot r$, $\coth r$. Those statements concern a single submanifold and characterise the free boundary condition; Theorem \ref{thm:defect} concerns the parametric field of a one-parameter family and computes the \emph{inhomogeneous} defect, whose vanishing is equivalent to criticality of the radius.

It should be said plainly that \eqref{eq:defect} is nothing but the derivative of the free boundary condition along the family, a computation standard in the bifurcation theory of free boundary and capillary problems: the content of Theorem \ref{thm:defect} is not the differentiation but what survives it. The differentiation produces a mixed term $A_a(w,\eta)$ with $w$ tangent to $\partial\Sigma$, and it is only because a free boundary on a totally umbilic barrier is a line of curvature (Lemma \ref{lem:curvline}) that this term vanishes identically, with no symmetry assumption and in every dimension. What survives is a single scalar multiple of $A_a(\eta,\eta)$, and the multiple is exactly $r'(a)$, so that one scalar derivative controls the entire boundary defect. Since $A_a(\eta,\eta)$ is nowhere zero for annuli in three-dimensional space form balls (Lemma \ref{lem:bdrycurv}), the implication becomes an equivalence, and the dimension count of Theorem \ref{thm:exactnul} becomes possible. This is what separates the dictionary from the familiar heuristic by which a failure of local injectivity of the radius forces a degeneration: that heuristic produces a degeneration only where $r$ fails to be locally injective, hence says nothing at a critical point of odd order, such as an inflection with $r'(a_*)=r''(a_*)=0$ and $r'''(a_*)\neq 0$, near which $r$ is still injective; it gives no converse; and it gives no dimension count.

\medskip
\noindent\textbf{On the conformal formulation of \cite{NZ1}.} A caveat must be recorded. In \cite[\S 2.5]{NZ1} the free boundary problem in $B_R\subset\Sph^3$ is recast, by a radial conformal change, as a \emph{weighted} free boundary minimal surface problem in the fixed hemisphere, the cap radius entering only through the weight. In that picture the barrier does not move, and one implication of the dictionary --- that a critical radius produces a nontrivial kernel --- is immediate; the converse and the dimension count, which are what turns the nullity into an exact indicator function, are not. We are explicit about this in Remark \ref{rem:conformal}, where the comparison belongs.

\medskip
\noindent\textbf{Scope of the claims.} Theorem A(iii) asserts a jump of the nullity at critical points of the radius, not a bifurcation; at a nondegenerate critical point the expected local picture within the family is a fold, the kernel element is the tangent vector of the family itself, and no new solutions are produced (Remark \ref{rem:nobif}). A degenerate annulus is explicit in the same sense in which the critical catenoid in the Euclidean unit ball is explicit. The \emph{ambient} family $\{\Sigma_a\}_{a\in(0,1/2)}$ from which it is drawn --- whose members are, with isolated exceptions, \emph{non}-degenerate, by Theorem \ref{thm:main}(iv) --- is in closed form, consisting of elementary functions and the single quadrature $\beta(a,\cdot)$ of \eqref{eq:dO-omega}; the cap radius itself is free even of that quadrature, since Proposition \ref{prop:Rclosed} gives
\[
\cot R(a)=-\,\frac{a\sin\bigl(2s_1(a)\bigr)}{\tfrac12+a\cos\bigl(2s_1(a)\bigr)},
\]
so that $R$ is an elementary function of $a$ and of the single scalar $s_1(a)$; and one member is then singled out by a transcendental scalar equation. For the critical catenoid that equation is $t\tanh t=1$ for the boundary parameter; here it is $R'(a_*)=0$, which by Proposition \ref{prop:Rprime} is the elementary relation \eqref{eq:critical} between $a_*$, $s_1(a_*)$ and $s_1'(a_*)$, and which by Theorem \ref{thm:dictionary} is equivalent to the vanishing of the Robin defect of the parametric field, $\partial_\eta\varphi_{a_*}=\cot(R_*)\,\varphi_{a_*}$ on $\partial\Sigma_{a_*}$. Within the family, $\Sigma_{a_*}$ is a local maximizer of both the cap radius and the area, the two maxima occurring at the same parameter by the first-variation identity \eqref{eq:areavar} of Proposition \ref{prop:area}; beyond these, no further extremal property of $\Sigma_{a_*}$ is claimed. Finally, degeneration and non-uniqueness are related but distinct phenomena, and a degenerate annulus is in general not a member of a non-uniqueness pair: the pairs of Theorem A(ii) at a noncritical value of the radius consist of annuli with \emph{vanishing} rotationally invariant even nullity, whereas a degenerate annulus of maximal cap radius sits at the extreme radius, where the two branches produced by the intermediate value theorem coalesce. It is the spectral witness of the fold, and the fold is what forces the failure of injectivity.

\medskip
\noindent\textbf{Organization.} Section \ref{sec:prelim} fixes conventions and records the two boundary lemmas that we need. Section \ref{sec:identity} proves Theorem B and its calibrations. Section \ref{sec:dictionary} develops the dictionary, the exact nullity identity and the structure at a critical point. Section \ref{sec:caps} contains the analysis of the de Oliveira family and the proof of Theorem A. Section \ref{sec:remarks} establishes the hyperbolic half of the dichotomy and collects the scope of the results and open problems.

\section{Conventions and preliminaries}\label{sec:prelim}

Throughout, $\M^{n+1}_\kappa$ denotes the complete simply connected space form of constant sectional curvature $\kappa\in\{-1,0,1\}$, with metric $\langle\cdot,\cdot\rangle$ and Levi-Civita connection $\bconn$. We fix $p_0\in\M^{n+1}_\kappa$, write $\rho=\dist(p_0,\cdot)$, and let $B(r)$ denote the closed geodesic ball of radius $r$ centred at $p_0$, with $0<r<\pi$ when $\kappa=1$. When $\kappa=1$ we also call $B(r)$ a \emph{spherical cap}; the radius is allowed to exceed $\tfrac\pi2$, in which case the cap contains more than a hemisphere and, by \eqref{eq:hess}, its boundary sphere has $A=\cot(r)\,g<0$ with respect to the outward normal, so that the cap is not convex. We set
\[
\sn(t)=\sinh t,\ t,\ \sin t,\qquad \ct(t)=\coth t,\ \tfrac1t,\ \cot t\qquad(\kappa=-1,0,1).
\]
On $B(r)\setminus\{p_0\}$ the Hessian comparison identity is an equality in space forms:
\begin{equation}\label{eq:hess}
\Hess\rho=\ct(\rho)\bigl(\langle\cdot,\cdot\rangle-d\rho\otimes d\rho\bigr).
\end{equation}
In particular the geodesic sphere $S(r)=\partial B(r)$ is totally umbilic.

Let $\Sigma^n$ be a compact manifold with nonempty boundary and let $\Phi:\Sigma\to B(r)$ be a two-sided minimal immersion with $\Phi(\partial\Sigma)\subset S(r)$ meeting $S(r)$ orthogonally along $\partial\Sigma$; we call such a $\Phi$ a \emph{free boundary minimal immersion}. We fix a smooth unit normal $\nu$ and adopt the conventions
\[
A(X,Y)=\langle\bconn_X\nu,Y\rangle,\qquad H=\tr_g A,
\]
so that geodesic spheres with outward normal satisfy $A=\ct(r)\,g$. We write $\eta$ for the outward unit conormal of $\partial\Sigma$ in $\Sigma$; the free boundary condition is equivalent to $\eta=\partial_\rho$ along $\partial\Sigma$, equivalently to $\langle\nu,\partial_\rho\rangle=0$ there. With the opposite convention $\mathrm{II}(X,Y)=\langle\bconn_XY,\nu\rangle=-A(X,Y)$, all identities below hold with $A$ replaced by $-\mathrm{II}$.

The second variation of volume for free boundary variations is governed by the quadratic form (see \cite[\S 5]{Med23}, \cite[\S 2.3]{NZ1})
\begin{equation}\label{eq:secondvar}
Q_r(u)=\int_\Sigma\Bigl(|\nabla u|^2-\bigl(|A|^2+n\kappa\bigr)u^2\Bigr)-\ct(r)\int_{\partial\Sigma}u^2,
\end{equation}
where we used $\Ric_{\M}(\nu,\nu)=n\kappa$. The associated operator and boundary operator are
\[
L=\Delta+|A|^2+n\kappa,\qquad \mathcal{B}_r u=\partial_\eta u-\ct(r)\,u,
\]
and integration by parts gives $Q_r(u)=-\int_\Sigma uLu+\int_{\partial\Sigma}u\,\mathcal{B}_ru$.

\begin{definition}\label{def:kernel}
The \emph{Jacobi--Robin kernel} of a free boundary minimal hypersurface $\Sigma\subset B(r)$ is
\[
\Ker(\Sigma,r)=\{u\in C^\infty(\Sigma):\ Lu=0,\ \mathcal{B}_ru=0\},
\]
and $\nul(\Sigma):=\dim\Ker(\Sigma,r)$. Elements of $\Ker(\Sigma,r)$ are called \emph{Jacobi--Robin fields}. If $Z$ is a Killing field of $\M^{n+1}_\kappa$ whose flow preserves $B(r)$, then $\langle Z,\nu\rangle\in\Ker(\Sigma,r)$ (Corollary \ref{cor:killing}); such elements are called \emph{Killing-induced}.
\end{definition}

The following classical lemma is the geometric reason why Theorem B requires no symmetry.

\begin{lemma}[The free boundary is a line of curvature]\label{lem:curvline}
Let $\Sigma\subset B(r)$ be a free boundary hypersurface, not necessarily minimal. Then $A(w,\eta)=0$ for every $w$ tangent to $\partial\Sigma$.
\end{lemma}

\begin{proof}
Differentiate $\langle\nu,\partial_\rho\rangle=0$ along $w\in T\partial\Sigma$:
\[
0=\langle\bconn_w\nu,\partial_\rho\rangle+\langle\nu,\bconn_w\partial_\rho\rangle=A(w,\eta)+\Hess\rho(w,\nu),
\]
where we used that $\bconn_w\nu$ is tangent to $\Sigma$ and that $\partial_\rho=\eta$ on $\partial\Sigma$. By \eqref{eq:hess} and $\langle w,\nu\rangle=\langle\nu,\partial_\rho\rangle=0$ we get $\Hess\rho(w,\nu)=0$.
\end{proof}

The same umbilicity of $S(r)$ has a tangential counterpart, which we record because of its specialization to rotational surfaces.

\begin{lemma}[The free boundary is umbilic in $\Sigma$]\label{lem:bdryumbilic}
Let $\Sigma\subset B(r)$ be a free boundary hypersurface, not necessarily minimal. Then, for all $v,w$ tangent to $\partial\Sigma$,
\begin{equation}\label{eq:bdryumbilic}
\langle\nabla^\Sigma_vw,\eta\rangle=-\ct(r)\,\langle v,w\rangle ;
\end{equation}
that is, $\partial\Sigma$ is umbilic as a hypersurface of $\Sigma$, with second fundamental form $\ct(r)\,g_{\partial\Sigma}$ with respect to the outward conormal $\eta$.
\end{lemma}

\begin{proof}
Since $\langle\nu,\eta\rangle=0$, the normal component of $\bconn_vw$ does not contribute, so $\langle\nabla^\Sigma_vw,\eta\rangle=\langle\bconn_vw,\eta\rangle$. Along $\partial\Sigma$ one has $\eta=\partial_\rho$ and $\langle w,\partial_\rho\rangle=0$; differentiating the latter along $v$,
\[
0=\langle\bconn_vw,\partial_\rho\rangle+\langle w,\bconn_v\partial_\rho\rangle=\langle\bconn_vw,\eta\rangle+\Hess\rho(v,w),
\]
and $\Hess\rho(v,w)=\ct(r)\bigl(\langle v,w\rangle-\langle v,\partial_\rho\rangle\langle w,\partial_\rho\rangle\bigr)=\ct(r)\langle v,w\rangle$ by \eqref{eq:hess}.
\end{proof}

\begin{corollary}\label{cor:ctb}
Let $n=2$ and let $\Sigma\subset B(r)\subset\M^3_\kappa$ be a rotational free boundary surface whose induced metric is $ds^2+b(s)^2d\theta^2$. Then, at a free boundary parameter $s_0$ at which the outward conormal is $\eta=\partial_s$,
\begin{equation}\label{eq:ctb}
\ct(r)=\frac{b'(s_0)}{b(s_0)} .
\end{equation}
\end{corollary}

\begin{proof}
The field $T=b^{-1}\partial_\theta$ is a unit tangent to $\partial\Sigma$, and the only Christoffel symbol needed is $\Gamma^s_{\theta\theta}=-b\,b'$, whence $\nabla^\Sigma_TT=-\tfrac{b'}{b}\,\partial_s$ and $\langle\nabla^\Sigma_TT,\eta\rangle=-b'(s_0)/b(s_0)$. Comparison with \eqref{eq:bdryumbilic} gives \eqref{eq:ctb}.
\end{proof}

Lemmas \ref{lem:curvline} and \ref{lem:bdryumbilic} are elementary and are stated here only because they are used, respectively, in the proof of Theorem \ref{thm:defect} and, through Corollary \ref{cor:ctb}, in \S\ref{sec:caps}.

We shall also use the following nonvanishing property, which is due to Naff--Zhu.

\begin{lemma}[{cf.\ \cite[Cor.~4.4]{Santos}, \cite[Prop.~4.1 and Cor.~4.2]{NZ1}}]\label{lem:bdrycurv}
Let $\Sigma$ be a free boundary minimal annulus in a geodesic ball $B(r)\subset\M^3_\kappa$. Then $|A|$ does not vanish on $\Sigma$; along $\partial\Sigma$ the conormal $\eta$ is a principal direction, $A(\eta,\eta)\neq 0$, and $A(\eta,\eta)$ has the same strict sign on both boundary circles.
\end{lemma}

Indeed, by Lemma \ref{lem:curvline} the frame $\{\eta,w\}$ diagonalizes $A$ along $\partial\Sigma$, so $|A|^2=2A(\eta,\eta)^2$ there, and the nonvanishing of $|A|$ for minimal annuli in space form balls is \cite[Cor.~4.4]{Santos}, where it is obtained for free boundary surfaces of constant mean curvature in space form balls, an annulus being characterised there by the absence of umbilical points; for minimal surfaces the two conditions coincide, since $H=0$ forces an umbilical point to be a zero of $A$. The same statement is proved in \cite[Prop.~4.1]{NZ1} by the Hopf differential method; the constancy of the sign of $A(\eta,\eta)$ is \cite[Cor. 4.2]{NZ1}, which is stated there for $\Sph^3$ but whose proof, using only \cite[Prop. 4.1]{NZ1} and the conformality of the parametrization, applies in every space form. For the family studied in \S\ref{sec:caps} both conclusions are reproved by a direct computation in \eqref{eq:Aclosed}, so that section is independent of \cite{NZ1} on this point.

\subsection*{Variable-radius families}

\begin{definition}\label{def:family}
A \emph{variable-radius family} is a smooth map $\Phi:I\times\Sigma\to\M^{n+1}_\kappa$, with $I\subset\R$ an interval and $\Sigma$ a compact manifold with boundary each of whose connected components has nonempty boundary, such that every $\Phi_a:=\Phi(a,\cdot)$ is a free boundary minimal immersion in $B(r(a))$, all balls being centred at the fixed point $p_0$, with $r:I\to(0,\infty)$ smooth and $r<\pi$ when $\kappa=1$. We write
\[
V_a=\partial_a\Phi_a,\qquad \varphi_a=\langle V_a,\nu_a\rangle\in C^\infty(\Sigma),
\]
for the variation field and its normal component, the unit normals $\nu_a$ being chosen smoothly in $a$. We call $\varphi_a$ the \emph{parametric field} of the family.
\end{definition}

All boundary components of $\Sigma$ are mapped into the single sphere $S(r(a))$; $\Sigma$ need not be connected, symmetric, or two-dimensional.

\begin{proposition}\label{prop:jacobi}
For every $a\in I$ one has $L\varphi_a=0$ on $\Sigma$.
\end{proposition}

\begin{proof}
For a smooth family of immersions with variation field $V_a$ of normal component $\varphi_a$, the linearization of the mean curvature is $\partial_aH(\Phi_a)=-L\varphi_a+V_a^{T}(H)$, where $V_a^T$ is the tangential part of $V_a$ (see e.g.\ \cite[\S 5]{Med23}). Since $H(\Phi_a)\equiv 0$ for every $a$, both $\partial_aH(\Phi_a)$ and $V_a^T(H)$ vanish, whence $L\varphi_a=0$. The conclusion is insensitive to the sign convention in the linearization, since only the vanishing of both sides is used.
\end{proof}

\begin{remark}[Independence of the parametrization]\label{rem:reparam}
Both sides of the identity of Theorem B are unchanged if $\Phi_a$ is replaced by $\Phi_a\circ\chi_a$, for a smooth family of diffeomorphisms $\chi_a$ of $\Sigma$ preserving $\partial\Sigma$. Such a replacement changes $V_a$ by a vector tangent to $\Sigma$, and along $\partial\Sigma$ by a vector tangent to $\partial\Sigma$; since $\nu$ is normal to $\Sigma$ and $\eta$ is normal to $\partial\Sigma$ within $\Sigma$, neither $\varphi_a=\langle V_a,\nu_a\rangle$, viewed as a function on $\Sigma$ up to the reparametrization, nor the conormal component $\langle V_a,\eta\rangle$ is affected. This freedom is used in \S\ref{sec:caps} to bring a family defined on an $a$-dependent domain into the form required by Definition \ref{def:family}.
\end{remark}

\section{The Robin defect identity}\label{sec:identity}

\begin{theorem}[Robin defect identity; Theorem B]\label{thm:defect}
Let $\{\Phi_a\}_{a\in I}$ be a variable-radius family as in Definition \ref{def:family}. Then, for every $a\in I$, on $\partial\Sigma$,
\begin{equation}\label{eq:defect}
\mathcal{B}_{r(a)}\varphi_a=\partial_\eta\varphi_a-\ct(r(a))\,\varphi_a=r'(a)\,A_a(\eta,\eta).
\end{equation}
\end{theorem}

\begin{proof}
Fix $a$ and drop the subscripts. All computations take place along $\partial\Sigma$, and $D_a$ denotes the pullback connection along the curves $a\mapsto\Phi_a(x)$.

\emph{Step 1: the conormal component of $V$.} Differentiating in $a$ the boundary constraint $\rho(\Phi_a(x))=r(a)$, valid for every $a$ and every $x\in\partial\Sigma$, gives $\langle\partial_\rho,V\rangle=r'(a)$. Since $\partial_\rho=\eta$ along the free boundary, the orthogonal decomposition of $V$ into normal, boundary-tangential and conormal parts reads
\begin{equation}\label{eq:decomp}
V=\varphi\,\nu+w+r'(a)\,\eta,\qquad w\in T\partial\Sigma.
\end{equation}

\emph{Step 2: differentiating the free boundary condition.} The identity $\langle\nu_a,\partial_\rho\circ\Phi_a\rangle=0$ holds on $\partial\Sigma$ for every $a$; differentiating in $a$,
\begin{equation}\label{eq:diffFB}
0=\langle D_a\nu,\partial_\rho\rangle+\langle\nu,\bconn_V\partial_\rho\rangle.
\end{equation}
By \eqref{eq:hess} and $\langle\nu,\partial_\rho\rangle=0$,
\begin{equation}\label{eq:term2}
\langle\nu,\bconn_V\partial_\rho\rangle=\Hess\rho(V,\nu)=\ct(r)\bigl(\langle V,\nu\rangle-\langle V,\partial_\rho\rangle\langle\nu,\partial_\rho\rangle\bigr)=\ct(r)\,\varphi.
\end{equation}
For the first term, let $e$ be tangent to $\Sigma$, extended as a coordinate field. Differentiating $\langle\nu,\Phi_{a*}e\rangle=0$ in $a$ and using the symmetry $D_a(\Phi_{a*}e)=\bconn_eV$ gives $\langle D_a\nu,e\rangle=-\langle\nu,\bconn_eV\rangle$. Writing $V=\varphi\nu+V^T$ and using $\langle\nu,V^T\rangle\equiv 0$ on $\Sigma$,
\[
\langle\nu,\bconn_eV\rangle=\partial_e\varphi+\langle\nu,\bconn_eV^T\rangle=\partial_e\varphi-\langle\bconn_e\nu,V^T\rangle=\partial_e\varphi-A(e,V^T).
\]
Both sides of the last identity are $C^\infty(\Sigma)$-linear in $e$ at fixed $a$, so it holds for every tangent field, in particular for the conormal $\eta$, which itself depends on $a$. Taking $e=\eta$ and recalling $\partial_\rho=\eta$ on $\partial\Sigma$,
\begin{equation}\label{eq:term1}
\langle D_a\nu,\partial_\rho\rangle=-\partial_\eta\varphi+A(\eta,V^T).
\end{equation}
Substituting \eqref{eq:term2} and \eqref{eq:term1} into \eqref{eq:diffFB} and using \eqref{eq:decomp},
\[
\partial_\eta\varphi-\ct(r)\,\varphi=A(\eta,V^T)=A(\eta,w)+r'(a)\,A(\eta,\eta)=r'(a)\,A(\eta,\eta),
\]
the mixed term vanishing by Lemma \ref{lem:curvline}.
\end{proof}

\begin{remark}[Capillary form of the Robin defect identity]\label{rem:kinematic}
Minimality is not used in the proof of \eqref{eq:defect}: the identity is a kinematic statement about an arbitrary family of free boundary hypersurfaces with boundaries on concentric geodesic spheres, and it enters spectral theory only through Proposition \ref{prop:jacobi}. The same differentiation applies verbatim at constant contact angle. Suppose that $\langle\nu_a,\partial_\rho\rangle=\cos\gamma$ along $\partial\Sigma$ for a constant $\gamma\in(0,\tfrac\pi2]$, so that $\partial_\rho=\sin\gamma\,\eta+\cos\gamma\,\nu$ there, the free boundary case being $\gamma=\tfrac\pi2$. The proof of Lemma \ref{lem:curvline} applies with the constant $\cos\gamma$ in place of $0$ and yields $\sin\gamma\,A(w,\eta)=0$, hence again $A(w,\eta)=0$; this is the capillary form of Joachimsthal's theorem, cf.\ \cite[Lem. 2.5]{NZ1}. Steps 1 and 2 of the proof of Theorem \ref{thm:defect} then give $\sin\gamma\,\langle V,\eta\rangle=r'(a)-\cos\gamma\,\varphi_a$ and $\Hess\rho(V,\nu)=\ct(r(a))\bigl(\varphi_a-r'(a)\cos\gamma\bigr)$, and eliminating $\langle V,\eta\rangle$ one obtains
\[
\partial_\eta\varphi_a-\Bigl(\frac{\ct(r(a))}{\sin\gamma}-\cot\gamma\,A_a(\eta,\eta)\Bigr)\varphi_a=\frac{r'(a)}{\sin\gamma}\Bigl(A_a(\eta,\eta)-\ct(r(a))\cos\gamma\Bigr),
\]
which reduces to \eqref{eq:defect} at $\gamma=\tfrac\pi2$ and whose left-hand side is exactly the Robin operator $\partial_\eta-q$ of the capillary second variation, with $q=\frac{1}{\sin\gamma}k^{S}(\bar\nu,\bar\nu)-\cot\gamma\,A_a(\eta,\eta)$ and $k^{S}=\ct(r)\,g$, cf.\ \cite[\S 2.3]{NZ1}. Nor need the angle be constant along the family. If $\gamma$ is a smooth function of $a$ with values in $(0,\tfrac\pi2]$, everything above applies at each fixed $a$, while differentiating $\langle\nu_a,\partial_\rho\rangle=\cos\gamma(a)$ produces the additional term $-\sin\gamma\,\gamma'(a)$ on the right; eliminating $\langle V,\eta\rangle$ and dividing by $\sin\gamma$ as before, one obtains the \emph{capillary form of the Robin defect identity}
\begin{equation}\label{eq:defect-cap}
\partial_\eta\varphi_a-\Bigl(\frac{\ct(r(a))}{\sin\gamma}-\cot\gamma\,A_a(\eta,\eta)\Bigr)\varphi_a=\frac{r'(a)}{\sin\gamma}\Bigl(A_a(\eta,\eta)-\ct(r(a))\cos\gamma\Bigr)+\gamma'(a)
\end{equation}
on $\partial\Sigma$. The left-hand side is unchanged, hence still the Robin operator of the capillary second variation; the right-hand side is now a linear functional of the pair $(r'(a),\gamma'(a))$ of derivatives of the constraint, the angle entering with coefficient $1$ and the radius with one that varies along $\partial\Sigma$; Identity \eqref{eq:defect} is the case $\gamma\equiv\tfrac\pi2$ of \eqref{eq:defect-cap}.
\\
As in Theorem \ref{thm:defect}, the ambient enters \eqref{eq:defect-cap} only through the Hessian of the distance function; the identity therefore extends beyond space forms, see Remark \ref{rem:warped}.

\noindent The capillary setting is recorded here only for completeness; no application of \eqref{eq:defect-cap} is developed in this paper.
\end{remark}

\begin{remark}[What the ambient hypothesis really is]\label{rem:warped}
The ambient geometry enters the proofs of Lemma  \ref{lem:curvline}, Lemma  \ref{lem:bdryumbilic}, Corollary  \ref{cor:ctb} and Theorem  \ref{thm:defect} only through the umbilicity identity \eqref{eq:hess}, and \eqref{eq:hess} does not require constant curvature. Consider, on the punctured ball $B(r_0)\setminus\{p_0\}$,
\[
g=d\rho^{2}+h(\rho)^{2}g_{\Sph^{n}} ,
\]
with $g_{\Sph^n}$ the round metric of curvature one and $h$ smooth and positive on $(0,r_0)$; if moreover $h$ extends to a smooth odd function with $h'(0)=1$, then $g$ closes up smoothly at $p_0$ and $\rho$ is the distance from $p_0$.
Since the radial curves are unit-speed geodesics and
$\Hess\rho=\tfrac12\mathcal L_{\partial_\rho}g$, while
$\mathcal L_{\partial_\rho}\bigl(h^{2}g_{\Sph^{n}}\bigr)=2hh'\,g_{\Sph^{n}}$,
one obtains
\[
\Hess\rho=\frac{h'(\rho)}{h(\rho)}\bigl(g-d\rho\otimes d\rho\bigr),
\]
so that $S(r)$ is totally umbilic with principal curvature $h'(r)/h(r)$. The four statements above therefore hold in this generality, the role of $\ct(r)$ being played by $h'(r)/h(r)$; in particular \eqref{eq:defect} reads
\[
\partial_\eta\varphi_a-\frac{h'(r(a))}{h(r(a))}\,\varphi_a
=r'(a)\,A_a(\eta,\eta)\qquad\text{on }\partial\Sigma ,
\]
the space forms being the cases $h(\rho)=\sin\rho,\ \rho,\ \sinh\rho$. The same applies to the capillary computation of Remark \ref{rem:kinematic}, which uses the ambient only through \eqref{eq:hess} as well: both identities there hold with $\ct(r)$ replaced by $h'(r)/h(r)$, and $k^{S}=(h'/h)\,g$. The linearisation used in the proof of Proposition  \ref{prop:jacobi} is the general one, $L=\Delta_\Sigma+|A|^{2}+\Ric_{\M}(\nu,\nu)$, which reduces to the operator of \S\ref{sec:prelim} through $\Ric_{\M}(\nu,\nu)=n\kappa$; here one retains $\Ric_{\M}(\nu,\nu)$, so that the zeroth-order term of $L$ is no longer constant, and \eqref{eq:secondvar} is to be read with $\ct(r)$ replaced by $h'(r)/h(r)$. With these substitutions the proof of Theorem  \ref{thm:dictionary} goes through unchanged, and so does that of Theorem \ref{thm:exactnul}: the spherical symmetry of $g$ provides orientation-reversing ambient isometries fixing $p_0$, so that \eqref{eq:reflection} remains meaningful, and the potential $|A_a|^{2}+\Ric_{\M}(\nu_a,\nu_a)$ is rotationally invariant and $s$-even, being
preserved by the rotations and by the reflection $\sigma$; the Sturm--Liouville reduction in mode zero is therefore unaffected, no constancy of the curvature term being used (see \S\ref{sec:dictionary}).

The fibre plays no role in the computation of $\Hess\rho$ above. Indeed $\mathcal L_{\partial_\rho}\bigl(h^{2}g_{F}\bigr)=2hh'\,g_{F}$ for an arbitrary Riemannian $n$-manifold $(F,g_{F})$, so on $I\times F$ with $g=d\rho^{2}+h(\rho)^{2}g_{F}$ and $h$ smooth and positive the radial curves are again unit-speed geodesics, the displayed formula for $\Hess\rho$ holds
unchanged, and the leaves $\{\rho=r\}$ are totally umbilic with principal curvature $h'(r)/h(r)$; the four statements above therefore persist here as well, for a family whose boundaries lie on the leaves, Definition \ref{def:family} being read with $B(r)$ replaced by the sublevel set
$\{\rho\le r\}$. Taking $F=\R^{n}$ flat with $h(\rho)=e^{\rho}$, and $F=\Hyp^{n}$ of curvature $-1$ with $h(\rho)=\cosh\rho$, one recovers the horospherical and the equidistant foliations of $\Hyp^{n+1}$, of principal curvatures $1$ and $\tanh\rho$. Theorem \ref{thm:exactnul}, by contrast, does not survive a change of fibre: its proof uses the rotational invariance and the reflection supplied by the round fibre, and for a general $F$ there is neither a Fourier decomposition nor an even sector, so that $\Ker_0^{\mathrm{ev}}$ is not defined. The spectral content of the identity is nonetheless unaffected: the linearisation of Proposition \ref{prop:jacobi} is insensitive to the fibre, so $L\varphi_a=0$ still holds, and $\varphi_a$ is a Jacobi--Robin field if and only if $r'(a)=0$ or $A_a(\eta,\eta)\equiv 0$ on $\partial\Sigma$. What is lost is the exact dimension count, together with the identification of the Killing-induced fields of Lemma \ref{lem:equivariance} and Proposition \ref{prop:killnul}, both of which rest on the round fibre and the pole.

One hypothesis does not transfer. Lemma  \ref{lem:bdrycurv} rests on the holomorphicity of the Hopf differential, which requires the vanishing of the Codazzi curvature term $\overline R(\cdot,\cdot,\cdot,\nu)$ --- automatic in constant curvature, but not in the present generality; so $A_a(\eta,\eta)\not\equiv 0$ has to be verified rather than deduced. For rotational families the hypothesis at least acquires a transparent geometric meaning: for $n\geq 2$, minimality forces $A(\eta,\eta)=-(n-1)k$ with $k$ the principal curvature in the rotational directions, so $A_a(\eta,\eta)$ vanishes exactly at the totally geodesic points of $\partial\Sigma_a$, and the hypothesis fails only when every boundary point is totally geodesic --- as happens for a totally geodesic disc through $p_0$. For the two families of this paper it is verified directly, in \eqref{eq:Aclosed} and \eqref{eq:Aclosed-mori}.

What the defect identity requires of the ambient is therefore not constant curvature, but that the barrier family be umbilic.
\end{remark}

\subsection*{Calibrations}

\begin{corollary}[Zero-defect case: isometric families]\label{cor:killing}
If the family is generated by a flow of ambient isometries, then $r'\equiv 0$ and \eqref{eq:defect} reads $\mathcal{B}_{r}\varphi_a=0$: normal components of Killing fields preserving the ball are Jacobi--Robin fields. This recovers, in particular, the Jacobi functions with boundary condition $\partial u/\partial\eta=\coth R\cdot u$, $\eta$ the outward conormal, used in \cite[eq. (26)]{LX18}.
\end{corollary}

\begin{corollary}[Euclidean degeneration]\label{cor:euclid}
Let $\kappa=0$, let $\Phi$ be a free boundary minimal immersion in the unit ball, and let $\Phi_t=t\,\Phi$ with $t>0$, so that $r(t)=t$. At $t=1$ one has $V=X$, the position field, and $\varphi=\langle X,\nu\rangle$, and \eqref{eq:defect} becomes
\[
\partial_\eta\langle X,\nu\rangle-\langle X,\nu\rangle=A(\eta,\eta)\qquad\text{on }\partial\Sigma.
\]
Since $X=\eta$ along the free boundary, $\langle X,\nu\rangle=0$ there, and the identity reduces to the classical support function relation $\partial_\eta\langle X,\nu\rangle=A(\eta,\eta)$ of \cite{FS16,AN16,Dev19,Tra20}.
\end{corollary}

\begin{remark}[Why the Euclidean radius carries no spectral information]\label{rem:scaling}
Along the Euclidean dilation family the parametric field vanishes identically on $\partial\Sigma$, so it satisfies a Dirichlet, not a Robin, condition. Consistently, under $x\mapsto\lambda x$ the pair $(L,\mathcal{B}_r)$ transforms covariantly, $(|A|^2,\ct)\mapsto(\lambda^{-2}|A|^2,\lambda^{-1}\ct)$; the Jacobi--Robin spectrum is merely rescaled by $\lambda^{-2}$, so the kernel, and in particular the nullity, is invariant along the family. All the content of \eqref{eq:defect} therefore lies in $\kappa\neq 0$.
\end{remark}

\begin{proposition}[Integral obstruction]\label{prop:pairing}
Let $\{\Phi_a\}$ be a variable-radius family. Then for every $a\in I$ and every $\psi\in\Ker(\Sigma_a,r(a))$,
\begin{equation}\label{eq:pairing}
r'(a)\int_{\partial\Sigma}\psi\,A_a(\eta,\eta)=0 .
\end{equation}
In particular, if $r'(a)\neq 0$, every Jacobi--Robin field of $\Sigma_a$ is $L^2(\partial\Sigma)$-orthogonal to $A_a(\eta,\eta)$.
\end{proposition}

\begin{proof}
Here $\psi$ solves $Lu=0$ by Definition \ref{def:kernel} and $\varphi_a$ by Proposition \ref{prop:jacobi}, so Green's identity gives
\[
0=\int_\Sigma\bigl(\psi L\varphi_a-\varphi_a L\psi\bigr)=\int_{\partial\Sigma}\bigl(\psi\,\partial_\eta\varphi_a-\varphi_a\,\partial_\eta\psi\bigr).
\]
Inserting $\partial_\eta\varphi_a=\ct(r(a))\varphi_a+r'(a)A_a(\eta,\eta)$ from \eqref{eq:defect} and $\partial_\eta\psi=\ct(r(a))\psi$ from $\mathcal{B}_{r(a)}\psi=0$, the terms containing $\ct(r(a))$ cancel and \eqref{eq:pairing} follows.
\end{proof}

Proposition \ref{prop:pairing} is the general form of the pairing argument employed in \cite{P1}. Combined with Lemma \ref{lem:bdrycurv} it gives at once, in the spherical case, the following.

\begin{corollary}\label{cor:signchange}
Let $\{\Phi_a\}$ be a variable-radius family of free boundary minimal annuli in geodesic balls of $\Sph^3$, with $\Sigma$ connected, and let $a$ be such that $r'(a)\neq 0$. Then every nonzero $\psi\in\Ker(\Sigma_a,r(a))$ changes sign on $\partial\Sigma$.
\end{corollary}

\begin{proof}
By Lemma \ref{lem:bdrycurv} the function $A_a(\eta,\eta)$ has a strict constant sign on $\partial\Sigma$. If $\psi$ did not change sign on $\partial\Sigma$, then $\psi\,A_a(\eta,\eta)$ would have a constant sign there, so \eqref{eq:pairing} together with $r'(a)\neq 0$ would force $\psi\,A_a(\eta,\eta)\equiv 0$, hence $\psi\equiv 0$ on $\partial\Sigma$. The Robin condition would then give $\partial_\eta\psi=\ct(r(a))\psi=0$ on $\partial\Sigma$ as well, so $\psi$ would solve $L\psi=0$ with vanishing Cauchy data along the whole of $\partial\Sigma$. Attach an exterior collar to $\Sigma$, extend the induced metric and the potential $|A|^2+n\kappa$ smoothly across $\partial\Sigma$, and extend $\psi$ by zero: for every test function $\phi$ supported in the enlarged surface, integration by parts on $\Sigma$ together with the vanishing of $\psi$ and $\partial_\eta\psi$ on $\partial\Sigma$ shows that the extension lies in $H^1$ and satisfies the equation weakly across $\partial\Sigma$, hence is smooth by elliptic regularity. It is therefore a solution of a second order elliptic equation which vanishes on an open set; unique continuation \cite{Aro57}, applied along a chain of overlapping balls in the connected surface $\Sigma$, then forces $\psi\equiv 0$, a contradiction.
\end{proof}

\section{The degeneration dictionary}\label{sec:dictionary}

\begin{theorem}[Dictionary]\label{thm:dictionary}
Let $\{\Phi_a\}$ be a variable-radius family and let $a\in I$ be such that $A_a(\eta,\eta)\not\equiv 0$ on $\partial\Sigma$. Then
\[
\varphi_a\in\Ker(\Sigma_a,r(a))\qquad\Longleftrightarrow\qquad r'(a)=0 .
\]
\end{theorem}

\begin{proof}
The identity $L\varphi_a=0$ always holds. If $r'(a)=0$, then $\mathcal{B}_{r(a)}\varphi_a=0$ pointwise by \eqref{eq:defect}, so $\varphi_a\in\Ker(\Sigma_a,r(a))$. Conversely, if $\mathcal{B}_{r(a)}\varphi_a\equiv 0$, then $r'(a)A_a(\eta,\eta)\equiv 0$ on $\partial\Sigma$; since $r'(a)$ is a number and $A_a(\eta,\eta)\not\equiv 0$, necessarily $r'(a)=0$.
\end{proof}

\begin{corollary}[Spectral rigidity forces monotonicity]\label{cor:monotone}
Let $\{\Phi_a\}_{a\in I}$ be a variable-radius family with $\varphi_a\not\equiv 0$ and $A_a(\eta,\eta)\not\equiv 0$ for every $a\in I$. If $\varphi_a\notin\Ker(\Sigma_a,r(a))$ for every $a\in I$, then $r'$ never vanishes on $I$ and $r$ is strictly monotone.
\end{corollary}

\begin{proof}
Immediate from Theorem \ref{thm:dictionary} and the continuity of $r'$ on the interval $I$.
\end{proof}

\begin{remark}[Degenerate parametric fields]\label{rem:degenerate}
The hypothesis $\varphi_a\not\equiv 0$ is genuinely needed and has a transparent meaning. If $\varphi_a\equiv 0$ and $A_a(\eta,\eta)\not\equiv 0$, then \eqref{eq:defect} forces $r'(a)=0$, so $V_a$ is everywhere tangent to $\Sigma_a$ and, along the boundary, tangent to $\partial\Sigma$; the first variation of the induced metric is then $\partial_ag_a=\mathcal{L}_{V_a^T}g_a$, so the family is stationary to first order modulo reparametrization. In concrete situations $\varphi_a\not\equiv 0$ is checked by direct computation, as we do in Lemma \ref{lem:phinonzero} for the spherical family; for the Mori family in $\Hyp^3$ it follows from the constant nonzero Wronskian recorded in \cite[Lem. 11.1]{P1}.
\end{remark}

\begin{proposition}[First variation of the area]\label{prop:area}
For a variable-radius family $\{\Phi_a\}_{a\in I}$ in the sense of Definition \ref{def:family}, write $|\Sigma_a|$ for the $n$-dimensional volume of $\Sigma_a$ and $|\partial\Sigma_a|$ for the $(n-1)$-dimensional volume of its boundary. Then
\begin{equation}\label{eq:areavar}
\frac{d}{da}\,|\Sigma_a|=r'(a)\,|\partial\Sigma_a|.
\end{equation}
Since $|\partial\Sigma_a|>0$, the critical points of $a\mapsto|\Sigma_a|$ are exactly the critical points of $r$.
\end{proposition}

\begin{proof}
Each $\Sigma_a$ is minimal, so the first variation of volume along the family is carried entirely by the boundary,
\[
\frac{d}{da}\,|\Sigma_a|=\int_{\partial\Sigma_a}\langle V_a,\eta\rangle\,dA_\partial,
\]
with $V_a=\partial_a\Phi_a$ the variation field, $\eta$ the outward unit conormal, and $dA_\partial$ the induced boundary volume element. By \eqref{eq:decomp}, $\langle V_a,\eta\rangle=r'(a)$ on $\partial\Sigma_a$, so the integrand equals the constant $r'(a)$ and \eqref{eq:areavar} follows. The final assertion is immediate from $|\partial\Sigma_a|>0$.
\end{proof}

\subsection*{Rotational annuli in three-dimensional balls}

We now specialize to $n=2$. Let $\{\Phi_a\}$ be a variable-radius family of annuli in $\M^3_\kappa$ which is \emph{rotational}: there is a geodesic $\ell$ through $p_0$ such that $\Phi_a$ is equivariant with respect to the group $SO(2)$ of rotations about $\ell$, and $(s,\theta)$ denote respectively the arclength of the profile and the rotation angle. Every $u\in C^\infty(\Sigma)$ splits into Fourier modes in $\theta$; \emph{mode zero} means rotationally invariant. We say that the family is \emph{reflection-symmetric} if there is an orientation-reversing ambient isometry $\sigma$ of $\M^3_\kappa$ fixing $p_0$ with
\begin{equation}\label{eq:reflection}
\sigma\circ\Phi_a(s,\theta)=\Phi_a(-s,\theta)\qquad\text{for every }a,s,\theta.
\end{equation}
We denote by $\Ker_0^{\mathrm{ev}}(\Sigma_a)$ the space of rotationally invariant, $s$-even Jacobi--Robin fields of $\Sigma_a$.

\begin{lemma}[Parity of the parametric field]\label{lem:even}
For a reflection-symmetric rotational variable-radius family of annuli in $\M^3_\kappa$, the parametric field $\varphi_a$ is rotationally invariant and even in $s$.
\end{lemma}

\begin{proof}
Rotational invariance of $\varphi_a$ follows from the equivariance of the family and of the smoothly chosen $\nu_a$.

We claim that $d\sigma\,\nu_a(s,\theta)=\nu_a(-s,\theta)$. The normal $\nu_a$ is determined by the convention that a positively oriented frame of $T\Sigma_a$ followed by $\nu_a$ is a positively oriented frame of $T\M^3_\kappa$. Let $\iota(s,\theta)=(-s,\theta)$; then $\iota$ reverses the orientation of $\Sigma$, since $\iota^*(ds\wedge d\theta)=-\,ds\wedge d\theta$, and $\sigma$ reverses the orientation of the ambient space by hypothesis. Let $(e_1,e_2)$ be a positively oriented frame at $(s,\theta)$ and set $f_i:=\Phi_{a*}d\iota(e_i)$, a frame of the tangent plane at the image of $(-s,\theta)$. Then $(\Phi_{a*}e_1,\Phi_{a*}e_2,\nu_a(s,\theta))$ is positively oriented, so its image under $d\sigma$, which by \eqref{eq:reflection} equals $(f_1,f_2,d\sigma\,\nu_a(s,\theta))$, is negatively oriented. On the other hand $(d\iota(e_1),d\iota(e_2))$ is a negatively oriented frame at $(-s,\theta)$, so $(f_1,f_2,\nu_a(-s,\theta))$ is negatively oriented as well. The two normals therefore coincide.

Differentiating \eqref{eq:reflection} in $a$ gives $d\sigma\,(\partial_a\Phi_a)(s,\theta)=(\partial_a\Phi_a)(-s,\theta)$. Since $\sigma$ is an isometry,
\[
\varphi_a(-s)=\langle(\partial_a\Phi_a)(-s),\nu_a(-s)\rangle=\langle d\sigma\,\partial_a\Phi_a(s),d\sigma\,\nu_a(s)\rangle=\langle\partial_a\Phi_a(s),\nu_a(s)\rangle=\varphi_a(s).\qedhere
\]
\end{proof}

\begin{theorem}[Exact rotationally invariant even nullity]\label{thm:exactnul}
Let $\{\Phi_a\}$ be a reflection-symmetric rotational variable-radius family of annuli in $\M^3_\kappa$ such that, for every $a$, $\varphi_a\not\equiv 0$ and $A_a(\eta,\eta)\neq 0$ on $\partial\Sigma$. Then, for every $a$,
\[
\dim\Ker_0^{\mathrm{ev}}(\Sigma_a)=
\begin{cases}
1,&r'(a)=0,\\
0,&r'(a)\neq 0;
\end{cases}
\]
equivalently, $a\mapsto\dim\Ker_0^{\mathrm{ev}}(\Sigma_a)$ is the indicator function of the critical set $\{r'=0\}$ of the radius map.
\end{theorem}

\begin{proof}
By Remark \ref{rem:reparam} neither the parametric field nor its conormal derivative, viewed as data on the image surface, is affected if $\Phi_a$ is precomposed with a smooth family of boundary-preserving diffeomorphisms; we may therefore work in the arclength gauge. Write the induced metric of $\Sigma_a$ as $ds^2+b(s)^2d\theta^2$ on $[-s_0,s_0]\times\Sph^1$, with $b>0$ and $s_0=s_0(a)$. In mode zero the Jacobi equation $Lu=0$ reduces to the Sturm--Liouville equation
\[
(b\,u')'+b\,q\,u=0,\qquad q=|A|^2+2\kappa .
\]
By \eqref{eq:reflection} the isometry $\sigma$ maps $\Sigma_a$ to itself, inducing $s\mapsto-s$; hence $b$ and $q$ are even in $s$ and the equation is invariant under $s\mapsto-s$. Consequently the involution $\iota u(s):=u(-s)$ maps the two-dimensional solution space to itself, and in terms of the Cauchy data at $s=0$ it acts by $(u(0),u'(0))\mapsto(u(0),-u'(0))$; its $+1$-eigenspace, that is the space of even solutions, is exactly one-dimensional, cut out by $u'(0)=0$. Since $\varphi_a$ is a nontrivial even mode-zero solution by Proposition \ref{prop:jacobi} and Lemma \ref{lem:even}, that space is $\R\varphi_a$.

An element $c\,\varphi_a$ lies in $\Ker_0^{\mathrm{ev}}(\Sigma_a)$ if and only if it satisfies the Robin condition, that is, by \eqref{eq:defect}, if and only if $c\,r'(a)\,A_a(\eta,\eta)\equiv 0$ on $\partial\Sigma$. As $A_a(\eta,\eta)\neq 0$, this holds precisely when $c=0$ or $r'(a)=0$, which is the stated dimension count.
\end{proof}

\begin{remark}[On the conformal formulation of \cite{NZ1}, and what the defect adds]\label{rem:conformal}
In \cite[\S 2.5]{NZ1} the free boundary problem in $B_R\subset\Sph^3$ is recast, by a radial conformal change, as a \emph{weighted} free boundary minimal surface problem in the fixed hemisphere $(\Sph^3,\partial B_{\pi/2})$, the cap radius entering only through the weight $f_R$, and \cite[Lem. 2.6]{NZ1} identifies the two stability operators up to conjugation. In that picture the barrier does not move, and one implication of Theorem \ref{thm:dictionary} becomes immediate: if $R'(a_*)=0$ then $\tfrac{d}{da}f_{R(a)}$ vanishes at $a_*$, so the variation field of the family is, to first order, a variation \emph{within a single} weighted free boundary problem, and its normal component lies in the kernel of the corresponding index form. 
The reverse implication, and the exact dimension count that turns the nullity into an indicator function, are what the defect identity supplies and the conformal picture does not. Reading off $R'(a)=0$ from the mere membership of $\varphi_a$ in the kernel requires, in formulation (C), an
\emph{interior} nonvanishing --- that $\langle\nabla\rho,\nu_a\rangle$ is not
identically zero on $\Sigma_a$ --- which carries no information on the dimension of the kernel. Theorem \ref{thm:defect} instead concentrates the
entire obstruction on the boundary, as the single scalar multiple $r'(a)$ of $A_a(\eta,\eta)$; since $A_a(\eta,\eta)$ is nowhere zero there (Lemma \ref{lem:bdrycurv}), the implication becomes the equivalence of Theorem \ref{thm:dictionary}, and, through the one-dimensionality of the even mode-zero Jacobi space, the exact identity $\dim\Ker_0^{\mathrm{ev}}(\Sigma_a)=\mathbf{1}_{\{r'=0\}}(a)$ of Theorem \ref{thm:exactnul}. This is the sense in which the defect adds what the conformal reformulation, by itself, does not.
\end{remark}

\begin{lemma}[Ball-preserving Killing fields and Fourier modes]\label{lem:equivariance}
Let $r\in(0,\pi)$, and let $\Sigma\subset B(r)\subset\Sph^3$ be a surface with unit normal $\nu$, invariant under the group $SO(2)$ of rotations about a geodesic $\ell$ through $p_0$. Let $\mathfrak{k}_0$ denote the algebra of Killing fields of $\Sph^3$ whose flows preserve $B(r)$, and let $Z_0$ generate the rotations about $\ell$. Then $\mathfrak{k}_0$ is the isotropy algebra of $p_0$, isomorphic to $\mathfrak{so}(3)$, and it decomposes under the adjoint action of $SO(2)$ as $\mathfrak{k}_0=\R Z_0\oplus\mathfrak{m}$ with $\mathfrak{m}$ two-dimensional of weight one. Moreover $\langle Z_0,\nu\rangle\equiv 0$, and $\langle Z,\nu\rangle$ lies in the Fourier modes $|k|=1$ for every $Z\in\mathfrak{m}$. In particular the normal component of every element of $\mathfrak{k}_0$ has vanishing mode-zero part.
\end{lemma}

\begin{proof}
Let $Z\in\mathfrak{k}_0$. Each isometry $F$ in its flow preserves $B(r)$, hence also $\partial B(r)$, so it preserves the function $x\mapsto\dist(x,\partial B(r))$ on $B(r)$. Since $\rho$ is $1$-Lipschitz and the radial geodesic issuing from $x$ meets $\partial B(r)$ after length $r-\rho(x)$, that function equals $r-\rho$, whose unique maximum point is $p_0$; therefore $F(p_0)=p_0$, and $Z$ lies in the isotropy algebra of $p_0$, isomorphic to $\mathfrak{so}(3)$. Note that this argument is insensitive to whether $r$ exceeds $\tfrac\pi2$. Conversely, every Killing field fixing $p_0$ preserves each geodesic ball centred at $p_0$. Under the adjoint action of the axial $SO(2)$ that algebra decomposes as the line $\R Z_0$ together with a two-dimensional subspace $\mathfrak{m}$ of weight one. The flow of $Z_0$ preserves $\Sigma$ setwise, so $Z_0$ is tangent to $\Sigma$ and $\langle Z_0,\nu\rangle\equiv 0$. For $Z\in\mathfrak{m}$, let $\mathcal{R}_\alpha$ denote the ambient rotation by $\alpha$ about $\ell$; it preserves $\Sigma$ and its unit normal, so
\[
\langle Z,\nu\rangle\circ\mathcal{R}_\alpha=\langle \mathrm{Ad}_{\mathcal{R}_\alpha^{-1}}Z,\nu\rangle .
\]
Hence $Z\mapsto\langle Z,\nu\rangle$ is $SO(2)$-equivariant from the weight-one representation $\mathfrak{m}$ into $C^\infty(\Sigma)$, so its image lies in the weight-one isotypic component, that is, in the Fourier mode $|k|=1$.
\end{proof}

\begin{proposition}[Killing-induced nullity]\label{prop:killnul}
Let $\Sigma\subset B(r)\subset\Sph^3$ be a free boundary minimal annulus invariant under the group of rotations about a geodesic $\ell$ through $p_0$. Then $\nul(\Sigma)\geq 2$, the two elements being induced by the rotations about the geodesics through $p_0$ orthogonal to $\ell$, and lying in the modes $|k|=1$.
\end{proposition}

\begin{proof}
Those rotations fix $p_0$, hence preserve $B(r)$ by Lemma \ref{lem:equivariance}, and induce Jacobi--Robin fields by Corollary \ref{cor:killing}. Let $Z$ be a nonzero combination of their generators and suppose $\langle Z,\nu\rangle\equiv 0$. Then $Z$ is tangent to $\Sigma$; being tangent to $\partial B(r)$ as well, it is tangent to $\partial\Sigma$, so its flow preserves $\Sigma$. As $\Sigma$ is already invariant under the rotations about $\ell$, it would then be invariant under rotations about two distinct geodesics through $p_0$, hence under the group they generate, which is the full isotropy group $SO(3)$ of $p_0$. The orbits of that group are the distance spheres from $p_0$ together with the two fixed points $\pm p_0$, so the connected surface $\Sigma$, being a union of orbits and two-dimensional, would have to be a single distance sphere: a zero-dimensional orbit cannot occur in a union of orbits which is a two-dimensional manifold, and two distinct distance spheres are disjoint. No annulus is a distance sphere. Hence $Z\mapsto\langle Z,\nu\rangle$ is injective on the two-dimensional space of such generators; those generators span the subspace $\mathfrak{m}$ of Lemma \ref{lem:equivariance}, so its image lies in the modes $|k|=1$.
\end{proof}

\begin{theorem}[Structure at a critical point of the radius]\label{thm:fold}
In the situation of Theorem \ref{thm:exactnul}, let $a_*$ satisfy $r'(a_*)=0$. Then:
\begin{enumerate}
\item[\textup{(i)}] $\varphi_{a_*}$ is a nonzero rotationally invariant, $s$-even Jacobi--Robin field of $\Sigma_{a_*}$; in particular $\dim\Ker_0^{\mathrm{ev}}(\Sigma_{a_*})=1$ and, when $\kappa=1$ and the axis passes through $p_0$, $\nul(\Sigma_{a_*})\geq 3$ by Proposition \ref{prop:killnul};
\item[\textup{(ii)}] the sign of $A_a(\eta,\eta)$ on each connected component of $\partial\Sigma$ is independent of $a$, and it is the same on both components; consequently, if $r'$ changes sign at $a_*$, then the Robin defect $\mathcal{B}_{r(a)}\varphi_a=r'(a)A_a(\eta,\eta)$ reverses sign across $a_*$;
\item[\textup{(iii)}] if moreover $\kappa=1$ and the rotation axis is a geodesic through $p_0$, then $\varphi_{a_*}$ is not induced by any Killing field of $\Sph^3$ whose flow preserves $B(r(a_*))$.
\end{enumerate}
\end{theorem}

\begin{proof}
(i) By $r'(a_*)=0$ and \eqref{eq:defect} we get $\mathcal{B}_{r(a_*)}\varphi_{a_*}=0$; moreover $L\varphi_{a_*}=0$ by Proposition \ref{prop:jacobi} and $\varphi_{a_*}\not\equiv 0$ by hypothesis. The dimension count is Theorem \ref{thm:exactnul}. For the last assertion, the two Killing-induced fields of Proposition \ref{prop:killnul} lie in the modes $|k|=1$ while $\varphi_{a_*}$ lies in mode $0$, so the three are linearly independent.

(ii) The function $(a,x)\mapsto A_a(\eta,\eta)(x)$ is smooth on $I\times\partial\Sigma$ and nowhere zero by hypothesis, hence of constant sign on each connected component of $I\times\partial\Sigma$; since $I$ is an interval and each boundary circle is connected, its sign on a given circle does not depend on $a$. That the two circles carry the same sign follows from the reflection symmetry: with $p=(s_0,\theta)$ and $\bar p=(-s_0,\theta)$, \eqref{eq:reflection} gives $d\sigma\,\eta(p)=\eta(\bar p)$, because $d\iota\,\partial_s=-\partial_s$ turns the outward conormal at $p$ into the outward conormal at $\bar p$, while $d\sigma\,\nu_a(p)=\nu_a(\bar p)$ by the proof of Lemma \ref{lem:even}; as $\sigma$ is an isometry, $A_a(\eta,\eta)(\bar p)=A_a(\eta,\eta)(p)$. The assertion now follows from \eqref{eq:defect}.

(iii) By Lemma \ref{lem:equivariance}, applied with $r=r(a_*)\in(0,\pi)$, the normal component of every Killing field of $\Sph^3$ whose flow preserves $B(r(a_*))$ has vanishing mode-zero part, whereas $\varphi_{a_*}$ is a nonzero mode-zero field by (i).
\end{proof}

\begin{remark}[Degeneration is not bifurcation]\label{rem:nobif}
Theorem \ref{thm:fold} asserts that the nullity jumps at critical points of $r$; it does not assert that new solutions branch off there. At a critical point with $r''\neq 0$ the radius map is locally two-to-one on one side and has no preimage on the other, so within the family the picture is a fold and the kernel element is the tangent vector of the family itself, as in the classical situation of coaxial Euclidean catenoids at maximal separation, where the Jacobi kernel appears and no new catenoid is born. Any bifurcation statement requires additional hypotheses, such as transversality of an eigenvalue crossing or a kernel transversal to the branch, which we neither assume nor verify; and branching of non-rotational solutions off a rotational family is governed by the kernel in the Fourier modes $|k|\geq 2$, about which $r'(a)$ carries no information.
\end{remark}

\section{The rotational family in spherical caps}\label{sec:caps}

\subsection*{The family}

We use the parametrization of the do Carmo--Dajczer rotational minimal surfaces of $\Sph^3$ in the free boundary normalization of de Oliveira \cite{dO24}, to which it agrees term by term \cite[eq. (1)--(6)]{dO24}. Let $\Sph^3\subset\R^4$ be the unit sphere with the round metric, let $p_0=(1,0,0,0)$, and let $\theta$ denote the rotation angle in the $x^3x^4$-plane. For $a\in(0,\tfrac12)$ set
\begin{equation}\label{eq:dO-defs}
\lambda_a(s)=\tfrac12-a\cos 2s,\qquad z_a(s)=\bigl(\tfrac12+a\cos 2s\bigr)^{1/2},\qquad c(a)=\bigl(\tfrac14-a^2\bigr)^{1/2},
\end{equation}
\begin{equation}\label{eq:dO-omega}
\omega(a,t)=\frac{c(a)}{z_a(t)\,\lambda_a(t)},\qquad \beta(a,s)=\int_0^s\omega(a,t)\,dt,
\end{equation}
\begin{equation}\label{eq:dO-imm}
x_a(s)=\lambda_a(s)^{1/2}\cos\beta(a,s),\qquad y_a(s)=\lambda_a(s)^{1/2}\sin\beta(a,s),
\end{equation}
\begin{equation}\label{eq:dO-X}
X_a(s,\theta)=\bigl(x_a(s),\,y_a(s),\,z_a(s)\cos\theta,\,z_a(s)\sin\theta\bigr).
\end{equation}
Then $x_a^2+y_a^2+z_a^2=\lambda_a+(1-\lambda_a)=1$, so $X_a$ maps into $\Sph^3$, and by \cite[Prop. 2.1]{dO24} it is a complete minimal immersion of $\R\times\Sph^1$ into $\Sph^3$, parametrized by arclength in $s$; the induced metric is $ds^2+z_a(s)^2d\theta^2$. The generating curve satisfies the minimality equations
\begin{equation}\label{eq:dO-ode}
\begin{gathered}
\ddot x_a+\frac{\dot z_a}{z_a}\dot x_a+2x_a=0,\qquad
\ddot y_a+\frac{\dot z_a}{z_a}\dot y_a+2y_a=0,\\[2pt]
\ddot z_a+\frac{\dot z_a^2-1}{z_a}+2z_a=0 ;
\end{gathered}
\end{equation}
these are \cite[eq. (9)--(11)]{dO24}; the first two express $\Delta_{\Sigma_a}X_a^i+2X_a^i=0$ for the coordinates $x_a$ and $y_a$, which depend on $s$ alone, so that $\Delta_{\Sigma_a}=\partial_s^2+(\dot z_a/z_a)\partial_s$ on them.
Since $\cos 2s$ and $\omega(a,\cdot)$ are even in $s$, the function $\beta(a,\cdot)$ is odd, so $x_a$ and $z_a$ are even and $y_a$ is odd. Consequently the reflection $\sigma:x^2\mapsto-x^2$ of $\R^4$, which restricts to an orientation-reversing isometry of $\Sph^3$ fixing $p_0$, satisfies $\sigma\circ X_a(s,\theta)=X_a(-s,\theta)$: condition \eqref{eq:reflection} holds.

The second fundamental form of these annuli is available in closed form, which we record because it makes the standing hypotheses of Section \ref{sec:dictionary} directly verifiable for this family. By rotational symmetry the coordinate directions $\partial_s,\partial_\theta$ are principal, so by minimality the principal curvatures are $\pm\mu_a$ with $\mu_a^2=A_a(\partial_s,\partial_s)^2$ and $|A_a|^2=2\mu_a^2$. The induced metric is $ds^2+z_a^2\,d\theta^2$, whose Gauss curvature is $-\ddot z_a/z_a$, so the Gauss equation in $\Sph^3$ gives $-\ddot z_a/z_a=1-\mu_a^2$, that is $\mu_a^2=1+\ddot z_a/z_a$; eliminating $\ddot z_a$ by \eqref{eq:dO-ode} yields $\mu_a^2=(1-\dot z_a^2-z_a^2)/z_a^2$. On the other hand $z_a^2=\tfrac12+a\cos 2s$ and $z_a\dot z_a=-a\sin 2s$ give
\[
1-\dot z_a^2-z_a^2=\frac{z_a^2-a^2\sin^2 2s-z_a^4}{z_a^2}
=\frac{\bigl(\tfrac12+a\cos 2s\bigr)\bigl(\tfrac12-a\cos 2s\bigr)-a^2\sin^2 2s}{z_a^2}=\frac{c(a)^2}{z_a^2} .
\]
Therefore
\begin{equation}\label{eq:Aclosed}
A_a(\partial_s,\partial_s)^2=\frac{c(a)^2}{z_a^4},\qquad |A_a|^2=\frac{2\,c(a)^2}{z_a^4}\qquad\text{on all of }\Sigma_a .
\end{equation}
In particular $|A_a|$ never vanishes, and since $\eta=\pm\partial_s$ along the boundary, $A_a(\eta,\eta)^2=c(a)^2/z_a^{4}>0$ at every boundary point; being continuous and nowhere zero on the connected surface $\Sigma_a$, the function $A_a(\partial_s,\partial_s)$ has a constant strict sign, the same on both boundary circles. Thus for this family the conclusions of Lemma \ref{lem:bdrycurv} hold by direct computation, independently of \cite{NZ1}.

Set $C_a:=\int_0^{\pi}\omega(a,t)\,dt$. The two-sided bound and the limit recorded in the next lemma are all that the argument requires about $C_a$. They are contained in \cite[Lem. 2.5]{dO24}, where they are obtained through the period function of the associated Otsuki tori and hence rest on external monotonicity results for that period; we give in Lemma \ref{lem:period} a short self-contained proof which uses nothing beyond elementary calculus, so that no property of $C_a$ beyond \eqref{eq:Cbounds} and \eqref{eq:Climit} is used below.

\begin{lemma}[The period function]\label{lem:period}
For every $a\in[0,\tfrac12)$ put $w_a(\tau):=\bigl(c(a)^2+a^2\sin^2\tau\bigr)^{1/2}$, so that $w_a(\tau)\in[c(a),\tfrac12]$. Then
\begin{equation}\label{eq:Cclosed}
C_a=\frac{c(a)}{4}\int_0^{2\pi}\frac{\bigl(1+2w_a(\tau)\bigr)^{1/2}}{w_a(\tau)^2}\,d\tau .
\end{equation}
Consequently
\begin{equation}\label{eq:Cbounds}
\pi<C_a\leq\sqrt2\,\pi\qquad\text{for every }a\in\bigl[0,\tfrac12\bigr),
\end{equation}
with equality on the right if and only if $a=0$, and
\begin{equation}\label{eq:Climit}
\lim_{a\to(1/2)^-}C_a=\pi .
\end{equation}
\end{lemma}

\begin{proof}
\emph{Step 1: the closed form \eqref{eq:Cclosed}.} Substituting $\tau=2t$ in the definition of $C_a$ and writing
\[
p(\tau):=\tfrac12+a\cos\tau,\qquad q(\tau):=\tfrac12-a\cos\tau,
\]
both of which are positive because $a<\tfrac12$, we obtain $C_a=\tfrac{c}{2}\int_0^{2\pi}p^{-1/2}q^{-1}\,d\tau$, where $c:=c(a)$. The translation $\tau\mapsto\tau+\pi$ interchanges $p$ and $q$ and leaves the integral over a full period unchanged, so that also $C_a=\tfrac{c}{2}\int_0^{2\pi}q^{-1/2}p^{-1}\,d\tau$. Averaging the two expressions,
\[
C_a=\frac{c}{4}\int_0^{2\pi}\Bigl(\frac{1}{p^{1/2}q}+\frac{1}{q^{1/2}p}\Bigr)d\tau
=\frac{c}{4}\int_0^{2\pi}\frac{p^{1/2}+q^{1/2}}{pq}\,d\tau .
\]
Now $p+q=1$ and
\[
pq=\tfrac14-a^2\cos^2\tau=\bigl(\tfrac14-a^2\bigr)+a^2\sin^2\tau=c^2+a^2\sin^2\tau=w_a^2 ,
\]
whence $\bigl(p^{1/2}+q^{1/2}\bigr)^2=p+q+2(pq)^{1/2}=1+2w_a$ and, both summands being positive, $p^{1/2}+q^{1/2}=(1+2w_a)^{1/2}$. This is \eqref{eq:Cclosed}. Finally $c^2\leq w_a^2\leq c^2+a^2=\tfrac14$, which is the stated range of $w_a$.

\emph{Step 2: an elementary integral.} For $A>0$ and $A+B>0$, the periodicity and the symmetry of $\sin^2$ give $\int_0^{2\pi}(A+B\sin^2\tau)^{-1}d\tau=4\int_0^{\pi/2}(A+B\sin^2\tau)^{-1}d\tau$, and the substitution $u=\tan\tau$ turns the latter into
\[
4\int_0^{\infty}\frac{du}{A(1+u^2)+Bu^2}=4\int_0^{\infty}\frac{du}{A+(A+B)u^2}=\frac{2\pi}{\bigl(A(A+B)\bigr)^{1/2}} .
\]
Taking $A=c^2$ and $B=a^2$, so that $A+B=\tfrac14$, we get
\begin{equation}\label{eq:w2int}
\int_0^{2\pi}\frac{d\tau}{w_a(\tau)^2}=\frac{2\pi}{\bigl(c^2/4\bigr)^{1/2}}=\frac{4\pi}{c} .
\end{equation}

\emph{Step 3: the bounds \eqref{eq:Cbounds}.} Since $0<w_a\leq\tfrac12$ we have $1<(1+2w_a)^{1/2}\leq\sqrt2$ pointwise, with equality on the right at $\tau$ if and only if $w_a(\tau)=\tfrac12$, that is $a\cos\tau=0$. Inserting these two inequalities into \eqref{eq:Cclosed} and using \eqref{eq:w2int},
\[
\pi=\frac{c}{4}\cdot\frac{4\pi}{c}<C_a\leq\sqrt2\cdot\frac{c}{4}\cdot\frac{4\pi}{c}=\sqrt2\,\pi .
\]
Equality on the right forces $a\cos\tau=0$ for almost every $\tau$, hence $a=0$; conversely $w_0\equiv\tfrac12$, so $C_0=\sqrt2\,\pi$.

\emph{Step 4: the limit \eqref{eq:Climit}.} From $(1+2w)^{1/2}\leq 1+w$ for $w\geq 0$, together with \eqref{eq:Cclosed} and \eqref{eq:w2int},
\[
\pi<C_a\leq\pi+\frac{c}{4}\int_0^{2\pi}\frac{d\tau}{w_a(\tau)} .
\]
On $[0,\tfrac\pi2]$ one has $\sin\tau\geq 2\tau/\pi$, so, using again the symmetries of $\sin^2$,
\[
\int_0^{2\pi}\frac{d\tau}{w_a}=4\int_0^{\pi/2}\frac{d\tau}{\bigl(c^2+a^2\sin^2\tau\bigr)^{1/2}}
\leq 4\int_0^{\pi/2}\frac{d\tau}{\bigl(c^2+4\pi^{-2}a^2\tau^2\bigr)^{1/2}}
=\frac{2\pi}{a}\operatorname{arcsinh}\frac{a}{c} .
\]
Hence, for $a\in(\tfrac14,\tfrac12)$, using $a>\tfrac14$ and $a<\tfrac12$,
\[
\pi<C_a\leq\pi+\frac{\pi c}{2a}\operatorname{arcsinh}\frac{a}{c}\leq\pi+2\pi\,c\operatorname{arcsinh}\frac{1}{2c} .
\]
As $a\to(\tfrac12)^-$ we have $c=c(a)\to 0^+$, and $c\operatorname{arcsinh}\bigl(1/(2c)\bigr)\to 0$; therefore $C_a\to\pi$.
\end{proof}

Since $\omega(a,\pi-t)=\omega(a,t)$, the function $\omega(a,\cdot)$ is symmetric about $t=\tfrac\pi2$, so $\beta(a,\tfrac\pi2)=\tfrac12C_a$, and \eqref{eq:Cbounds} gives
\begin{equation}\label{eq:beta-half}
\beta\bigl(a,\tfrac\pi2\bigr)=\tfrac12\,C_a\in\Bigl(\frac\pi2,\frac{\pi}{\sqrt2}\Bigr]\subset\Bigl(\frac\pi2,\pi\Bigr)\qquad\text{for every }a\in\bigl[0,\tfrac12\bigr).
\end{equation}

\subsection*{The free boundary function}

For $X\in\Sph^3\setminus\{\pm p_0\}$ one has $\rho(X)=\arccos\langle X,p_0\rangle$ and
\[
\nabla\rho=-\frac{p_0-\langle p_0,X\rangle X}{\bigl(1-\langle p_0,X\rangle^2\bigr)^{1/2}} .
\]
Let $\nu$ be a unit normal of the immersion \eqref{eq:dO-X}. Since $\langle\nu,X\rangle=0$, the free boundary condition $\langle\nu,\nabla\rho\rangle=0$ at a parameter $s$ is equivalent to $\langle\nu,p_0\rangle=0$, that is, to the vanishing of the first component of $\nu$. A direct computation of the $4$-dimensional cross product of $X_a$, $\partial_sX_a$ and $\partial_\theta X_a$ shows that this first component equals $z_a\,G_a$ up to normalization, where
\begin{equation}\label{eq:Ga}
G_a(s):=y_a(s)\dot z_a(s)-z_a(s)\dot y_a(s).
\end{equation}
As $z_a>0$, the free boundary parameters are exactly the zeros of $G_a$.

The function $G_a$ is, up to a \emph{negative} factor, the function $f_a$ introduced in \cite[Lem. 3.1]{dO24}: one has $G_a=-f_a/\bigl(z_a\lambda_a^{1/2}\bigr)$, consistently with $G_a(0)=-z_a(0)<0$ and $f_a(0)>0$. Thus the free boundary condition takes there the form $f_a(s_1)=0$, which is the identity \eqref{eq:fbc-clean} established in the proof below, and the two auxiliary facts $s_1(a)<\tfrac\pi2$ and $x_a(s_1)<0$ recover, respectively, a step in the proof of \cite[Prop. 3.3]{dO24} and a part of \cite[Prop. 3.4]{dO24}; the embeddedness and containment statements themselves are recovered in Proposition \ref{prop:embcont} below. What the following lemma adds, and what the rest of the paper requires, is that the zero is \emph{simple} and \emph{unique} in $(0,\tfrac\pi2]$: this is not recorded in \cite{dO24}, it is what feeds the analytic implicit function theorem in Lemma \ref{lem:analytic} and makes $s_1$ a globally well-defined analytic branch, and without it Theorem \ref{thm:main}(iv) is unavailable.

\begin{lemma}\label{lem:s1}
Let $a\in[0,\tfrac12)$. Then $G_a$ has exactly one zero $s_1(a)$ in $(0,\tfrac\pi2]$, and $s_1(a)<\tfrac\pi2$; the zero is simple, with $G_a'(s_1(a))>0$. Moreover $\beta(a,s_1(a))\in\bigl(\tfrac\pi2,\tfrac12C_a\bigr)$ for $a\in(0,\tfrac12)$, and $s_1(0)=\pi/(2\sqrt2)$. Finally, $\dot x_a(s_1(a))<0$ for every $a\in(0,\tfrac12)$.
\end{lemma}

\begin{proof}
From $z_a^2=\tfrac12+a\cos2s$ we get $z_a\dot z_a=-a\sin 2s$, and from \eqref{eq:dO-defs}--\eqref{eq:dO-imm},
\begin{equation}\label{eq:ydot}
\dot y_a=\frac{a\sin 2s}{\lambda_a^{1/2}}\sin\beta+\lambda_a^{1/2}\,\omega\cos\beta,\qquad \lambda_a\,\omega=\frac{c(a)}{z_a}.
\end{equation}
At $s=0$ we have $\beta=0$, hence $y_a(0)=0$ and, by \eqref{eq:ydot} and $\lambda_a(0)=\tfrac12-a$,
\[
\dot y_a(0)=\lambda_a(0)^{1/2}\,\omega(a,0)=\frac{c(a)}{\lambda_a(0)^{1/2}z_a(0)}=\frac{\bigl[(\tfrac12-a)(\tfrac12+a)\bigr]^{1/2}}{(\tfrac12-a)^{1/2}(\tfrac12+a)^{1/2}}=1 .
\]
Therefore $G_a(0)=-z_a(0)<0$. At $s=\tfrac\pi2$ we have $\sin 2s=0$, so $\dot z_a(\tfrac\pi2)=0$ and, by \eqref{eq:ydot} and \eqref{eq:beta-half},
\[
G_a\bigl(\tfrac\pi2\bigr)=-z_a\bigl(\tfrac\pi2\bigr)\dot y_a\bigl(\tfrac\pi2\bigr)=-z_a\bigl(\tfrac\pi2\bigr)\lambda_a\bigl(\tfrac\pi2\bigr)^{1/2}\omega\bigl(a,\tfrac\pi2\bigr)\cos\bigl(\tfrac12C_a\bigr)>0
\]
for every $a\in[0,\tfrac12)$, since $\tfrac12C_a\in(\tfrac\pi2,\pi)$ by \eqref{eq:beta-half} and $\omega>0$. Hence $G_a$ has at least one zero in $(0,\tfrac\pi2)$.

Next, $G_a'=y_a\ddot z_a-z_a\ddot y_a$, and substituting \eqref{eq:dO-ode},
\[
G_a'=-\frac{y_a(\dot z_a^2-1)}{z_a}+\dot y_a\dot z_a .
\]
At a zero $s_*$ of $G_a$ one has $\dot y_a(s_*)=y_a(s_*)\dot z_a(s_*)/z_a(s_*)$, and substituting this the expression collapses to
\begin{equation}\label{eq:Gprime}
G_a'(s_*)=\frac{y_a(s_*)}{z_a(s_*)} .
\end{equation}
For $s\in(0,\tfrac\pi2]$ we have $\beta(a,s)\in(0,\tfrac12C_a]\subset(0,\pi)$ by \eqref{eq:beta-half} and $\omega>0$, hence $\sin\beta(a,s)>0$ and $y_a(s)>0$. Therefore $G_a'(s_*)>0$ at every zero $s_*\in(0,\tfrac\pi2]$: all zeros are simple upward crossings. Since $G_a(0)<0$, the function $G_a$ can cross zero only once in $(0,\tfrac\pi2]$, and the crossing occurs in $(0,\tfrac\pi2)$; we call it $s_1(a)$.

For the location of $\beta(a,s_1)$ we use the identity $G_a(s_1)=0$, that is, $\dot y_a=y_a\dot z_a/z_a$ at $s=s_1$. Using $\dot z_a=-a\sin2s/z_a$, $y_a=\lambda_a^{1/2}\sin\beta$ and \eqref{eq:ydot}, and multiplying by $\lambda_a^{1/2}$, we obtain, all functions being evaluated at $s=s_1$,
\[
a\sin (2s_1)\,\sin\beta+\lambda_a\,\omega\,\cos\beta=-\,\frac{a\,\lambda_a\,\sin\beta\,\sin (2s_1)}{z_a^{2}} .
\]
Moving the right-hand side to the left and using $z_a^2+\lambda_a=1$, the two terms carrying $\sin\beta$ combine into $a\sin(2s_1)\sin\beta\,z_a^{-2}$; using $\lambda_a\omega=c(a)/z_a$ we obtain
\begin{equation}\label{eq:fbc-clean}
a\,\sin\bigl(2s_1(a)\bigr)\,\sin\beta\bigl(a,s_1(a)\bigr)=-\,c(a)\,z_a\bigl(s_1(a)\bigr)\cos\beta\bigl(a,s_1(a)\bigr).
\end{equation}
For $a\in(0,\tfrac12)$ the left-hand side is strictly positive, because $s_1\in(0,\tfrac\pi2)$ and $\sin\beta(a,s_1)>0$; since $c(a)>0$ and $z_a>0$, we conclude $\cos\beta(a,s_1)<0$, that is $\beta(a,s_1)\in(\tfrac\pi2,\pi)$. As $\beta(a,\cdot)$ is strictly increasing and $s_1<\tfrac\pi2$, we also have $\beta(a,s_1)<\beta(a,\tfrac\pi2)=\tfrac12C_a$, which gives the stated interval.

The same two signs determine the side from which the annulus meets the geodesic sphere. Differentiating \eqref{eq:dO-imm} and using $\partial_s\beta=\omega$ and $\dot\lambda_a=2a\sin 2s$,
\[
\dot x_a=\frac{a\sin 2s}{\lambda_a^{1/2}}\cos\beta-\lambda_a^{1/2}\,\omega\,\sin\beta .
\]
At $s=s_1(a)$, with $a\in(0,\tfrac12)$, both summands are strictly negative: the first because $\sin(2s_1)>0$ and $\cos\beta(a,s_1)<0$, the second because $\lambda_a^{1/2}\omega>0$ and $\sin\beta(a,s_1)>0$. Hence $\dot x_a(s_1(a))<0$.

Finally, for $a=0$ we have $z_0\equiv 1/\sqrt2$, $\lambda_0\equiv 1/2$, $\omega(0,\cdot)\equiv\sqrt2$ and $\beta(0,s)=\sqrt2\,s$; hence $\dot z_0\equiv 0$, $y_0=\tfrac1{\sqrt2}\sin(\sqrt2 s)$, $\dot y_0=\cos(\sqrt2 s)$ and $G_0=-\tfrac1{\sqrt2}\cos(\sqrt2 s)$, whose first positive zero is $s_1(0)=\pi/(2\sqrt2)<\tfrac\pi2$.
\end{proof}

By Lemma \ref{lem:s1}, for $a\in(0,\tfrac12)$ the map $X_a$ restricted to $[-s_1(a),s_1(a)]\times\Sph^1$ is a free boundary minimal annulus whose two boundary circles lie on the geodesic sphere $S(R(a))$, where
\begin{equation}\label{eq:Rdef}
R(a)=\arccos x_a\bigl(s_1(a)\bigr)\in\Bigl(\frac\pi2,\pi\Bigr),
\end{equation}
the inclusion $R(a)>\tfrac\pi2$ being equivalent to $x_a(s_1(a))<0$, which follows from $\cos\beta(a,s_1)<0$ proved above; both circles lie on the same sphere because $x_a$ is even. We denote this annulus by $\Sigma_a$. That $\Sigma_a$ is moreover embedded and contained in $B(R(a))$ is due to de Oliveira \textup{\cite[Prop. 3.3 and Prop. 3.4]{dO24}}; we record short proofs, resting only on Lemma \ref{lem:s1}, so that this section is self-contained.

\begin{proposition}[Embeddedness and containment]\label{prop:embcont}
Let $a\in(0,\tfrac12)$. Then $X_a$ is injective on $[-s_1(a),s_1(a)]\times\Sph^1$, so that $\Sigma_a$ is an embedded annulus; and $x_a$ is strictly decreasing on $[0,s_1(a)]$, so that $\rho$ attains its maximum over $\Sigma_a$ exactly along $\partial\Sigma_a$. In particular $\Sigma_a\subset B(R(a))$.
\end{proposition}

\begin{proof}
\emph{Injectivity.} Suppose $X_a(s,\theta)=X_a(s',\theta')$ with $s,s'\in[-s_1,s_1]$. Comparing the norms of the first two coordinates gives $\lambda_a(s)=\lambda_a(s')$; dividing by $\lambda_a^{1/2}>0$ then gives $\bigl(\cos\beta(a,s),\sin\beta(a,s)\bigr)=\bigl(\cos\beta(a,s'),\sin\beta(a,s')\bigr)$, that is $\beta(a,s)-\beta(a,s')\in 2\pi\mathbb{Z}$. Now $\beta(a,\cdot)$ is odd and strictly increasing, since $\omega>0$, and $\beta(a,s_1)<\tfrac12 C_a\leq\pi/\sqrt2$ by Lemma \ref{lem:s1} and \eqref{eq:Cbounds}; hence $\beta(a,\cdot)$ maps $[-s_1,s_1]$ injectively onto an interval of length $2\beta(a,s_1)<\sqrt2\,\pi<2\pi$, so that $\beta(a,s)=\beta(a,s')$ and therefore $s=s'$. The last two coordinates then read $z_a(s)(\cos\theta,\sin\theta)=z_a(s)(\cos\theta',\sin\theta')$ with $z_a>0$, whence $\theta=\theta'$. An injective immersion of a compact manifold is an embedding.

\emph{Monotonicity of $x_a$.} Write $u:=x_a$. By \eqref{eq:dO-ode}, $u$ solves the linear equation $\ddot u+(\dot z_a/z_a)\dot u+2u=0$ on $\R$. If $u(s_c)=\dot u(s_c)=0$ for some $s_c$, then $u\equiv 0$ by uniqueness for that equation, contradicting $u(0)=\lambda_a(0)^{1/2}=(\tfrac12-a)^{1/2}>0$; hence at every critical point $s_c$ of $u$ one has $u(s_c)\neq 0$ and $\ddot u(s_c)=-2u(s_c)$. Moreover $\dot u(0)=0$, because $\sin 0=0$ and $\beta(a,0)=0$, so $\ddot u(0)=-2u(0)<0$ and $\dot u<0$ on some interval $(0,\varepsilon)$.

Suppose $\dot u$ had a zero in $(0,s_1]$ and let $s_c$ be the smallest one, so that $\dot u<0$ on $(0,s_c)$. If $u(s_c)>0$, then $\ddot u(s_c)=-2u(s_c)<0$, so $\dot u>0$ immediately to the left of $s_c$, contradicting $\dot u<0$ there. If $u(s_c)<0$, then $\cos\beta(a,s_c)<0$; since $s_c\leq s_1<\tfrac\pi2$ we have $\sin(2s_c)>0$, and $\beta(a,s_c)\in(0,\pi)$ gives $\sin\beta(a,s_c)>0$, so both summands of
\[
\dot x_a=\frac{a\sin 2s}{\lambda_a^{1/2}}\cos\beta-\lambda_a^{1/2}\,\omega\,\sin\beta
\]
are strictly negative at $s_c$, contradicting $\dot u(s_c)=0$. Since $u(s_c)=0$ is excluded, no such zero exists and $\dot x_a<0$ on $(0,s_1]$.

\emph{Containment.} Along the profile $\cos\rho=\langle X_a,p_0\rangle=x_a$, so $\rho$ is strictly increasing on $[0,s_1]$; as $x_a$ is even and $\rho$ does not depend on $\theta$, the maximum of $\rho$ over $\Sigma_a$ equals $\rho(s_1)=R(a)$ and is attained exactly on $\partial\Sigma_a$.
\end{proof}

\begin{remark}[On the containment]\label{rem:constrained}
The containment is not automatic here. For $r\leq\tfrac\pi2$, and for a surface known in advance to lie in the closed
hemisphere $\{\langle X,p_0\rangle\geq 0\}$, one may apply the maximum principle to the first coordinate $\langle X,p_0\rangle$, which satisfies $\Delta_{\Sigma}\langle X,p_0\rangle=-2\langle X,p_0\rangle$, and conclude that
a free boundary minimal surface with boundary on $S(r)$ lies in $B(r)$; cf.\ \cite[Rem. 2.3]{NZ1}. For $r>\tfrac\pi2$ the boundary value of that coordinate is negative and the geodesic sphere $S(r)$ is concave as seen from $p_0$, so the argument does not apply; the proof above replaces it by the sign information of Lemma \ref{lem:s1} propagated along the profile through \eqref{eq:dO-ode}. Proposition \ref{prop:embcont} is used, and only, where a geometric statement about the annuli as subsets of $\Sph^3$ is made: the containment in the assertion of Theorem \ref{thm:main}(ii) that the ball $B(\rho)$ \emph{contains} two annuli, and the embeddedness in the same item and in Corollary \ref{cor:Q66}. The identity of Theorem \ref{thm:defect}, the dictionary of Theorem \ref{thm:dictionary} and the nullity statements of Theorem \ref{thm:main}(iii)--(iv) use only that $\partial\Sigma_a$ meets $S(R(a))$ orthogonally \emph{and from inside}, in the sense that $\eta=\partial_\rho$ on both boundary circles; this is genuinely weaker, and is established in Remark \ref{rem:side}.
\end{remark}

\begin{remark}[The contact is from inside]\label{rem:side}
The free boundary convention of \S\ref{sec:prelim} requires $\eta=\partial_\rho$, not merely $\nu_a\perp\partial_\rho$, and the two are not equivalent: on a boundary circle met from outside the ball one has $\eta=-\partial_\rho$, and the computation of Theorem \ref{thm:defect} returns $\partial_\eta\varphi_a+\ct(r(a))\varphi_a=-r'(a)A_a(\eta,\eta)$, which is not the Robin operator $\mathcal{B}_{r(a)}$ of \eqref{eq:secondvar}. For the present family the correct side is forced by Lemma \ref{lem:s1}. Indeed $\cos\rho=x_a$ along the profile, so $\dot x_a=-\sin\rho\,\dot\rho$; since $\sin R(a)>0$ and $\dot x_a(s_1(a))<0$, we get $\dot\rho(s_1(a))>0$. At a free boundary parameter $\nu_a\perp\partial_\rho$, so $\nabla\rho$ is tangent to $\Sigma_a$ there; as $\partial_\theta\rho\equiv 0$ by rotational symmetry, $\nabla^{\Sigma_a}\rho=\dot\rho\,\partial_s$ has unit length, whence $\dot\rho(s_1(a))=1$ and $\eta=\partial_s=\partial_\rho$ at $s=s_1(a)$. Since $\rho$ is even in $s$ and the outward conormal at $s=-s_1(a)$ is $-\partial_s$, the same holds on the other circle. Thus $\{\Psi_a\}$ is a variable-radius family in the sense of Definition \ref{def:family}.
\end{remark}

\subsection*{Analyticity and the limits of the radius}

\begin{lemma}[Analyticity]\label{lem:analytic}
The maps $a\mapsto s_1(a)$ and $a\mapsto R(a)$ extend real-analytically to a neighbourhood of $[0,\tfrac12)$ inside $(-\tfrac12,\tfrac12)$; in particular they are real-analytic on $(0,\tfrac12)$ and continuous at $a=0$ with $R(0)=\tfrac\pi2$, so $R(a)\to\tfrac\pi2$ as $a\to 0^+$.
\end{lemma}

\begin{proof}
On $\{(a,t):|a|<\tfrac12,\ t\in\R\}$ the functions $\lambda_a(t)$ and $z_a(t)^2$ are real-analytic and bounded below by $\tfrac12-|a|>0$, so $\omega$ is real-analytic there. It extends holomorphically to a complex neighbourhood of each point, and the integration in \eqref{eq:dO-omega} runs over a compact interval, so Fubini and Morera give the joint real-analyticity of $\beta$; hence $x_a$, $y_a$, $z_a$ and $G_a$ are real-analytic jointly in $(a,s)$ on $(-\tfrac12,\tfrac12)\times\R$. By Lemma \ref{lem:s1}, $G_a(s_1(a))=0$ and $\partial_sG_a(s_1(a))>0$ for every $a\in[0,\tfrac12)$; the analytic implicit function theorem therefore yields a real-analytic branch through each point $(a,s_1(a))$, defined on a two-sided neighbourhood of that $a$ in $(-\tfrac12,\tfrac12)$, and by the uniqueness of the zero in $(0,\tfrac\pi2]$ these branches agree with $s_1$ on $[0,\tfrac12)$ and patch to a single analytic function there. Finally $|x_a(s_1)|\leq\lambda_a(s_1)^{1/2}\leq(\tfrac12+a)^{1/2}<1$, so $x_a(s_1(a))$ takes values in $(-1,1)$, where $\arccos$ is real-analytic; hence $R$ is real-analytic. For $a=0$ we have $\beta(0,s_1(0))=\sqrt2\cdot\pi/(2\sqrt2)=\tfrac\pi2$, so $x_0(s_1(0))=0$ and $R(0)=\tfrac\pi2$.
\end{proof}

\begin{lemma}[The radius returns to the hemisphere]\label{lem:return}
$\displaystyle\lim_{a\to(1/2)^-}R(a)=\frac\pi2 .$
\end{lemma}

\begin{proof}
Write $s_1=s_1(a)$. By Lemma \ref{lem:s1}, $\beta(a,s_1)\in\bigl(\tfrac\pi2,\tfrac12C_a\bigr)$. By \eqref{eq:Climit}, $\tfrac12C_a\to\tfrac\pi2$ as $a\to(\tfrac12)^-$, while $\tfrac12 C_a>\tfrac\pi2$ by \eqref{eq:Cbounds}; hence $\beta(a,s_1)\to\tfrac\pi2$ and $\cos\beta(a,s_1)\to 0$. Since
\[
0<\lambda_a(s_1)^{1/2}\leq\bigl(\tfrac12+a\bigr)^{1/2}<1,
\]
we get
\[
\bigl|\cos R(a)\bigr|=\bigl|x_a(s_1)\bigr|=\lambda_a(s_1)^{1/2}\,\bigl|\cos\beta(a,s_1)\bigr|\longrightarrow 0 .
\]
As $x_a(s_1)<0$ gives $R(a)\in(\tfrac\pi2,\pi)$, we conclude $R(a)\to(\tfrac\pi2)^+$.
\end{proof}

\subsection*{A closed form for the radius}

Corollary \ref{cor:ctb} removes the quadrature $\beta$ from the radius altogether.

\begin{proposition}[Closed form of the cap radius]\label{prop:Rclosed}
For every $a\in[0,\tfrac12)$,
\begin{equation}\label{eq:Rclosed}
\cot R(a)=\frac{\dot z_a\bigl(s_1(a)\bigr)}{z_a\bigl(s_1(a)\bigr)}=-\,\frac{a\sin\bigl(2s_1(a)\bigr)}{\tfrac12+a\cos\bigl(2s_1(a)\bigr)} .
\end{equation}
In particular $R$ is an elementary function of $a$ and of the single scalar $s_1(a)$.
\end{proposition}

\begin{proof}
The induced metric of $\Sigma_a$ is $ds^2+z_a(s)^2d\theta^2$ and the outward conormal at $s=s_1(a)$ is $\partial_s$, so Corollary \ref{cor:ctb}, applied with $\kappa=1$, $b=z_a$ and $r=R(a)$, gives the first equality. The second follows from $z_a\dot z_a=-a\sin 2s$ and $z_a^2=\tfrac12+a\cos 2s$.
\end{proof}

\begin{remark}\label{rem:Rclosed}
Identity \eqref{eq:Rclosed} yields at once two of the assertions of Theorem \ref{thm:main}(i), by a route independent of the location of $\beta(a,s_1(a))$. For $a\in(0,\tfrac12)$ Lemma \ref{lem:s1} gives $s_1(a)\in(0,\tfrac\pi2)$, hence $\sin(2s_1(a))>0$ and $\cot R(a)<0$; since $R(a)\in(0,\pi)$ this forces $R(a)>\tfrac\pi2$. At $a=0$ the numerator vanishes, so $\cot R(0)=0$ and $R(0)=\tfrac\pi2$. Likewise, by \eqref{eq:Rclosed} the limit $R(a)\to\tfrac\pi2$ as $a\to(\tfrac12)^-$ is equivalent to $a\sin(2s_1(a))/z_a(s_1(a))^2\to 0$.
\end{remark}

\begin{proposition}[Closed form of the derivative of the radius]\label{prop:Rprime}
For every $a\in(0,\tfrac12)$, writing $s_1=s_1(a)$ and $z=z_a(s_1)$,
\begin{equation}\label{eq:Rprime}
R'(a)=\frac{\tfrac12\sin(2s_1)+2a\bigl(z^2\cos(2s_1)+a\sin^2(2s_1)\bigr)\,s_1'(a)}{z^4+a^2\sin^2(2s_1)} ,
\end{equation}
whose denominator is strictly positive, and
\begin{equation}\label{eq:s1prime}
s_1'(a)=-\,\frac{z_a(s_1)}{y_a(s_1)}\,\bigl(\partial_aG_a\bigr)\bigl(s_1\bigr).
\end{equation}
In particular $a$ is a critical point of $R$ if and only if
\begin{equation}\label{eq:critical}
\tfrac12\sin(2s_1)+2a\bigl(z^2\cos(2s_1)+a\sin^2(2s_1)\bigr)\,s_1'(a)=0 .
\end{equation}
\end{proposition}

\begin{proof}
Put $F(a,s):=-a\sin(2s)\bigl/\bigl(\tfrac12+a\cos 2s\bigr)$, so that $\cot R(a)=F\bigl(a,s_1(a)\bigr)$ by \eqref{eq:Rclosed}, and abbreviate $z^2=\tfrac12+a\cos 2s=z_a(s)^2$. Differentiating $F$,
\[
\partial_aF=-\frac{\sin 2s}{2z^4},\qquad
\partial_sF=-\frac{2a\bigl(z^2\cos 2s+a\sin^2 2s\bigr)}{z^4},
\]
the first because $\bigl(\tfrac12+a\cos2s\bigr)-a\cos 2s=\tfrac12$. Since $s_1$ is differentiable by Lemma \ref{lem:analytic}, the chain rule gives
\[
\bigl(\cot R\bigr)'(a)=-\frac{1}{z^4}\Bigl[\tfrac12\sin(2s_1)+2a\bigl(z^2\cos(2s_1)+a\sin^2(2s_1)\bigr)s_1'(a)\Bigr],
\]
all functions being evaluated at $s=s_1(a)$. As $R(a)\in(0,\pi)$ one has $\bigl(\cot R\bigr)'=-R'/\sin^2R$, while \eqref{eq:Rclosed} gives
\[
\sin^2R=\bigl(1+\cot^2R\bigr)^{-1}=\frac{z^4}{z^4+a^2\sin^2(2s_1)} .
\]
Combining the three displays yields \eqref{eq:Rprime}, and \eqref{eq:critical} follows because the denominator of \eqref{eq:Rprime} is positive. Finally, \eqref{eq:s1prime} is the implicit differentiation of $G_a(s_1(a))=0$, legitimate by Lemma \ref{lem:analytic}, together with $\partial_sG_a(s_1)=y_a(s_1)/z_a(s_1)$ from \eqref{eq:Gprime} and $y_a(s_1)>0$ from the proof of Lemma \ref{lem:s1}.
\end{proof}

All the transcendence of the radius map is thus concentrated in the single scalar $s_1(a)$, and, through \eqref{eq:s1prime}, in the single quadrature $\partial_a\beta$; see Problem \ref{prob:caps}.

\subsection*{The parametric field}

We first bring the family into the form of Definition \ref{def:family}. Let $\Sigma=[-1,1]\times\Sph^1$ and set
\begin{equation}\label{eq:Psi}
\Psi_a(\tau,\theta):=X_a\bigl(s_1(a)\tau,\theta\bigr),\qquad a\in\bigl(0,\tfrac12\bigr).
\end{equation}
By Lemma \ref{lem:analytic} the map $(a,\tau,\theta)\mapsto\Psi_a(\tau,\theta)$ is real-analytic; every $\Psi_a$ is a free boundary minimal immersion in $B(R(a))$ with $\Psi_a(\partial\Sigma)\subset S(R(a))$; and $\Psi$ is rotational and reflection-symmetric in the sense of \eqref{eq:reflection}, with $\sigma:x^2\mapsto-x^2$ and $\tau\mapsto-\tau$. Thus $\{\Psi_a\}$ is a reflection-symmetric rotational variable-radius family, with radius function $R$. Its parametric field is
\[
\varphi_a=\bigl\langle\partial_a\Psi_a,\nu_a\bigr\rangle,\qquad \partial_a\Psi_a=\partial_aX_a+s_1'(a)\,\tau\,\partial_sX_a ,
\]
and since $\partial_sX_a$ is tangent to $\Sigma_a$, the correction $s_1'(a)\,\tau\,\partial_sX_a$ never contributes: $\varphi_a=\langle\partial_aX_a,\nu_a\rangle$ at every point, the derivative $\partial_aX_a$ being taken at fixed $s$ and evaluated at $s=s_1(a)\tau$.

\begin{lemma}[Nondegeneracy of the parametric field]\label{lem:phinonzero}
For every $a\in(0,\tfrac12)$ one has
\[
\bigl|\varphi_a(0)\bigr|=\frac{1}{2\,c(a)}=\frac{1}{2\bigl(\tfrac14-a^2\bigr)^{1/2}}>0 .
\]
In particular $\varphi_a\not\equiv 0$ for every $a\in(0,\tfrac12)$.
\end{lemma}

\begin{proof}
We compute at $(s,\theta)=(0,0)$. There $\beta(a,0)=0$, hence
\[
X_a(0,0)=\Bigl(\bigl(\tfrac12-a\bigr)^{1/2},\,0,\,\bigl(\tfrac12+a\bigr)^{1/2},\,0\Bigr).
\]
Since $\sin 2s$ vanishes at $s=0$, we have $\dot x_a(0)=0$ and $\dot z_a(0)=0$, while $\dot y_a(0)=1$ as computed in the proof of Lemma \ref{lem:s1}; hence
\[
\partial_sX_a(0,0)=(0,1,0,0),\qquad \partial_\theta X_a(0,0)=\bigl(0,0,0,z_a(0)\bigr).
\]
A unit normal at that point is orthogonal to $X_a(0,0)$, to $(0,1,0,0)$ and to $(0,0,0,z_a(0))$, hence has vanishing second and fourth components; writing it as $(\nu^1,0,\nu^3,0)$, the conditions $\nu^1x_a(0)+\nu^3z_a(0)=0$ and $(\nu^1)^2+(\nu^3)^2=1$ together with $x_a(0)^2+z_a(0)^2=1$ give
\[
\nu_a(0,0)=\pm\bigl(-z_a(0),\,0,\,x_a(0),\,0\bigr).
\]
Next, $\beta(a,0)=0$ and $\partial_a\beta(a,0)=0$, since $\beta(a,0)$ is an integral over a degenerate interval; therefore
\[
\partial_aX_a(0,0)=\Bigl(-\tfrac12\bigl(\tfrac12-a\bigr)^{-1/2},\ 0,\ \tfrac12\bigl(\tfrac12+a\bigr)^{-1/2},\ 0\Bigr).
\]
Consequently, up to the sign of $\nu_a$,
\[
\varphi_a(0)=\frac{1}{2}\frac{\bigl(\tfrac12+a\bigr)^{1/2}}{\bigl(\tfrac12-a\bigr)^{1/2}}+\frac{1}{2}\frac{\bigl(\tfrac12-a\bigr)^{1/2}}{\bigl(\tfrac12+a\bigr)^{1/2}}
=\frac{1}{2}\cdot\frac{\bigl(\tfrac12+a\bigr)+\bigl(\tfrac12-a\bigr)}{\bigl[\bigl(\tfrac12-a\bigr)\bigl(\tfrac12+a\bigr)\bigr]^{1/2}}=\frac{1}{2\,c(a)} .\qedhere
\]
\end{proof}

\subsection*{Proof of Theorem A}

\begin{lemma}[The core circle detects the parameter]\label{lem:core}
Let $a_1,a_2\in(0,\tfrac12)$ and suppose that $\Sigma_{a_1}$ and $\Sigma_{a_2}$ are isometric with their induced metrics. Then $a_1=a_2$. In particular, distinct parameters give non-congruent annuli.
\end{lemma}

\begin{proof}
The induced metric of $\Sigma_a$ is $ds^2+z_a(s)^2d\theta^2$ on $[-s_1(a),s_1(a)]\times\Sph^1$. For a curve joining $(s,\theta)$ to a point of $\partial\Sigma_a$, the length is at least $\int|\dot s|\geq s_1(a)-|s|$, with equality along the meridian; hence
\[
\dist_{\Sigma_a}\bigl((s,\theta),\partial\Sigma_a\bigr)=s_1(a)-|s| .
\]
Therefore $s_1(a)=\max_{\Sigma_a}\dist(\cdot,\partial\Sigma_a)$, and the set where this maximum is attained is the circle $\Gamma_a:=\{s=0\}$, whose length is $2\pi z_a(0)=2\pi\bigl(\tfrac12+a\bigr)^{1/2}$. Both quantities are determined by the isometry class of $(\Sigma_a,g_a)$, so an isometry $\Sigma_{a_1}\to\Sigma_{a_2}$ maps $\Gamma_{a_1}$ onto $\Gamma_{a_2}$ and preserves its length, giving $\bigl(\tfrac12+a_1\bigr)^{1/2}=\bigl(\tfrac12+a_2\bigr)^{1/2}$, that is $a_1=a_2$.
\end{proof}

\begin{theorem}[Theorem A]\label{thm:main}
Let $\{\Sigma_a\}_{a\in(0,1/2)}$ and $R$ be as above. Then:
\begin{enumerate}
\item[\textup{(i)}] $R$ is real-analytic on $(0,\tfrac12)$, $R>\tfrac\pi2$ there, and $R(a)\to\tfrac\pi2$ both as $a\to 0^+$ and as $a\to(\tfrac12)^-$. In particular $R$ is non-constant and non-monotone, and it attains its maximum at some interior point $a_*\in(0,\tfrac12)$, where $R'(a_*)=0$ and $R(a_*)>\tfrac\pi2$.
\item[\textup{(ii)}] For every $\rho\in\bigl(\tfrac\pi2,\max R\bigr)$ the set $\{a\in(0,\tfrac12):R(a)=\rho\}$ is finite and contains at least two elements; the corresponding annuli are embedded free boundary minimal annuli in $B(\rho)$ and are mutually non-congruent. In particular $R$ is not injective.
\item[\textup{(iii)}] Let $a_*\in(0,\tfrac12)$ satisfy $R'(a_*)=0$. Then $\varphi_{a_*}$ is a nonzero rotationally invariant, reflection-even Jacobi--Robin field of $\Sigma_{a_*}$; hence $\dim\Ker_0^{\mathrm{ev}}(\Sigma_{a_*})=1$ and, by Proposition \ref{prop:killnul}, $\nul(\Sigma_{a_*})\geq 3$. Moreover $\varphi_{a_*}$ is not induced by any Killing field of $\Sph^3$ whose flow preserves $B(R(a_*))$.
\item[\textup{(iv)}] The critical set $\{R'=0\}$ is discrete in $(0,\tfrac12)$ and finite in every compact subinterval, and $\dim\Ker_0^{\mathrm{ev}}(\Sigma_a)=0$ for every $a$ outside it.
\end{enumerate}
\end{theorem}

\begin{proof}
(i) Real-analyticity is Lemma \ref{lem:analytic}; the strict inequality $R>\tfrac\pi2$ is \eqref{eq:Rdef}, which rests on $\cos\beta(a,s_1(a))<0$ from Lemma \ref{lem:s1}; and the two limits are Lemma \ref{lem:analytic} and Lemma \ref{lem:return}. Extending $R$ continuously to $[0,\tfrac12]$ by the value $\tfrac\pi2$ at both endpoints, we obtain a continuous function on a compact interval which is strictly larger than its endpoint values in the interior; hence it attains its maximum at an interior point $a_*$, where $R'(a_*)=0$ and $R(a_*)>\tfrac\pi2$. Since $R$ takes values arbitrarily close to $\tfrac\pi2$ near both endpoints while $R(a_*)>\tfrac\pi2$, the map $R$ is neither monotone nor constant.

(ii) Fix $\rho\in(\tfrac\pi2,\max R)$. Since $R$ is continuous, tends to $\tfrac\pi2<\rho$ at both endpoints and satisfies $R(a_*)>\rho$, the intermediate value theorem provides $a_1\in(0,a_*)$ and $a_2\in(a_*,\tfrac12)$ with $R(a_1)=R(a_2)=\rho$. The set $E_\rho:=\{a:R(a)=\rho\}$ is closed in $(0,\tfrac12)$ and, by the two limits, contained in a compact subinterval of $(0,\tfrac12)$; since $R$ is real-analytic and non-constant, $R-\rho$ has isolated zeros, so $E_\rho$ is a compact set of isolated points and therefore finite. Embeddedness and containment in $B(\rho)$ are Proposition \ref{prop:embcont}, and non-congruence is Lemma \ref{lem:core}.

(iii) By \eqref{eq:Psi} the family $\{\Psi_a\}$ is a reflection-symmetric rotational variable-radius family in $\M^3_1=\Sph^3$; by Lemma \ref{lem:phinonzero}, $\varphi_a\not\equiv 0$ for every $a$; and by \eqref{eq:Aclosed}, $A_a(\eta,\eta)\neq 0$ on $\partial\Sigma$ with constant sign, in accordance with Lemma \ref{lem:bdrycurv}. Hence Theorem \ref{thm:exactnul} and Theorem \ref{thm:fold} apply with $r=R$, and (iii) is Theorem \ref{thm:fold}(i) and (iii).

(iv) By (i), $R$ is real-analytic and non-constant on $(0,\tfrac12)$, so $R'$ is real-analytic and not identically zero; its zeros are therefore isolated, and finite in number on every compact subinterval. The last assertion is Theorem \ref{thm:exactnul}.
\end{proof}

\begin{definition}[A degenerate annulus of maximal cap radius]\label{def:degannulus}
By Theorem \ref{thm:main}(i) the radius map attains its maximum at an interior parameter. Fix a maximizer $a_*\in(0,\tfrac12)$ and set $R_*:=\max_{(0,1/2)}R=R(a_*)>\tfrac\pi2$. We call $\Sigma_{a_*}$ \emph{a degenerate annulus of maximal cap radius} of the family, and $R_*$ its \emph{maximal cap radius}. The name is justified by Theorem \ref{thm:main}(iii): $\Sigma_{a_*}$ is an embedded free boundary minimal annulus, contained in the largest geodesic ball $B(R_*)$ reached by the family, whose Jacobi--Robin kernel contains a rotationally invariant, reflection-even element not induced by any ambient Killing field preserving the ball. Here \emph{degenerate} is understood modulo the ambient isometries preserving the ball, as is customary: by Proposition \ref{prop:killnul} every member of the family satisfies $\nul(\Sigma_a)\geq 2$ on account of Killing-induced fields alone, so that bare nontriviality of the kernel distinguishes nothing; what singles out $\Sigma_{a_*}$ is a kernel element which is \emph{not} of that form. The value $R_*$ is canonical, whereas the maximizer need not a priori be unique: by Lemma \ref{lem:core} two distinct maximizers would give non-congruent annuli, and whether this occurs is part of Problem \ref{prob:caps}. By Theorem \ref{thm:main}(iii) \emph{every} critical point of $R$ yields a degenerate annulus, maximal or not; by Theorem \ref{thm:main}(iv) these form a discrete set. No statement below depends on the uniqueness of the maximizer.
\end{definition}

\begin{corollary}[Variational characterization]\label{cor:areamax}
Along the family $\{\Sigma_a\}_{a\in(0,1/2)}$ the critical points of the area $a\mapsto|\Sigma_a|$ coincide with the critical points of $R$, at which the annulus is degenerate by Theorem \ref{thm:main}(iii)--(iv). Moreover a degenerate annulus of maximal cap radius $\Sigma_{a_*}$ is a strict local maximizer of the area within the family.
\end{corollary}

\begin{proof}
By Proposition \ref{prop:area} with $r=R$ one has $\frac{d}{da}|\Sigma_a|=R'(a)\,|\partial\Sigma_a|$ with $|\partial\Sigma_a|>0$, so the two critical sets coincide; the degeneracy at those parameters is Theorem \ref{thm:main}(iii). Let $a_*$ be a maximizer of $R$. By Theorem \ref{thm:main}(i) $R$ is real-analytic and non-constant, so the zeros of $R'$ are isolated; hence there is $\delta>0$ with $R'\neq 0$ on $(a_*-\delta,a_*)\cup(a_*,a_*+\delta)$, and the maximality of $R(a_*)$ forces $R'>0$ on the left interval and $R'<0$ on the right. Thus $\frac{d}{da}|\Sigma_a|=R'(a)\,|\partial\Sigma_a|$ changes sign from positive to negative at $a_*$, and $|\Sigma_{a_*}|$ is a strict local maximum.
\end{proof}

\begin{corollary}[Failure of the Killing hypothesis above the hemisphere]\label{cor:Q66}
There exists $\rho>\tfrac\pi2$ and an embedded free boundary minimal annulus $\Sigma\subset B(\rho)\subset\Sph^3$ whose Jacobi--Robin kernel contains an element that is not induced by any Killing field of $\Sph^3$ preserving $B(\rho)$. Consequently, the property that all Jacobi fields of an embedded free boundary minimal annulus in a geodesic ball of $\Sph^3$ be Killing-induced cannot hold for all radii in $(0,\pi)$.
\end{corollary}

\begin{proof}
Take $\rho=R_*$ and $\Sigma=\Sigma_{a_*}$ a degenerate annulus of maximal cap radius as in Definition \ref{def:degannulus}, and apply Theorem \ref{thm:main}(iii); embeddedness is Proposition \ref{prop:embcont}.
\end{proof}

\begin{remark}[Scope]\label{rem:scope}
Two points delimit Corollary \ref{cor:Q66}. First, the degenerate parameter $a_*$ lies in the regime $R>\tfrac\pi2$, strictly beyond the hemisphere; it is not in the range $R\in[0,\tfrac\pi2]$ on which the continuity and uniqueness programme of \cite[\S 6]{NZ1} operates. Question 6.6 of \cite{NZ1} carries no explicit restriction on $R$ in its own statement, but it appears in a section whose standing framework is $R\in[0,\tfrac\pi2]$ throughout, and \cite[Rem. 1.3]{NZ1} already records that uniqueness appears to fail above the hemisphere. We therefore read Corollary \ref{cor:Q66} not as a refutation of that question, but as the statement that its hypothesis does not survive the passage beyond $R=\tfrac\pi2$, while on $R\leq\tfrac\pi2$ the question is untouched by the present work. Second, the non-injectivity exhibited here occurs on the branch $a\in(0,\tfrac12)$, where $R>\tfrac\pi2$; it does not bear on the complementary branch $a\in(-\tfrac12,0)$, where $R<\tfrac\pi2$ by Proposition \ref{prop:subhemi} and where de Oliveira expects bijectivity and reports strong numerical evidence for it \cite[Rem. 3.5]{dO24}. The two branches are disjoint, and the picture that emerges is of a radius map that is injective below the hemisphere, as expected, and folded above it.
\end{remark}

\subsection*{The complementary branch}

Throughout this subsection $a\in(-\tfrac12,0)$, and $s_1(a)$, $R(a)$, $\Sigma_a$ and $\Psi_a$ are defined by \eqref{eq:Ga}, \eqref{eq:Rdef} and \eqref{eq:Psi} for such $a$ as well. The analysis of the present section transfers to this range with one sign reversal, and Theorem \ref{thm:exactnul} applies there too. Theorem \ref{thm:main} does not: the fold is established only on the branch $a>0$, and whether $R$ folds on the complementary branch is the content of Remark \ref{rem:dO35}.

\begin{proposition}[The sub-hemispherical branch]\label{prop:subhemi}
Let $a\in(-\tfrac12,0)$. Then $G_a$ has exactly one zero $s_1(a)$ in $(0,\tfrac\pi2]$, it satisfies $s_1(a)<\tfrac\pi2$ and is simple with $G_a'(s_1(a))>0$; one has $\beta(a,s_1(a))\in(0,\tfrac\pi2)$, hence $x_a(s_1(a))>0$ and $R(a)<\tfrac\pi2$; the contact is from inside, so that
$\eta=\partial_\rho$ on $\partial\Sigma_a$; the annulus $\Sigma_a$ is embedded and contained in $B(R(a))$; and $R$ is real-analytic on $(-\tfrac12,0)$.
Consequently $\{\Psi_a\}_{a\in(-1/2,0)}$ is a reflection-symmetric rotational variable-radius family satisfying the hypotheses of Theorem \ref{thm:exactnul}, and
\[
\dim\Ker_0^{\mathrm{ev}}(\Sigma_a)=\mathbf 1_{\{R'=0\}}(a),
\qquad a\in(-\tfrac12,0).
\]
\end{proposition}

\begin{proof}
Directly from \eqref{eq:dO-defs} and \eqref{eq:dO-omega} one has $c(-a)=c(a)$, $z_{-a}(t)=\lambda_a(t)^{1/2}$ and $\lambda_{-a}(t)=z_a(t)^2$, while $z_a(t+\tfrac\pi2)=\lambda_a(t)^{1/2}$ and $\lambda_a(t+\tfrac\pi2)=z_a(t)^2$; hence $\omega(-a,t)$ and $\omega(a,t+\tfrac\pi2)$ both equal $c(a)\,\lambda_a(t)^{-1/2}z_a(t)^{-2}$, so that
$\omega(-a,t)=\omega(a,t+\tfrac\pi2)$. Since $\omega(a,\cdot)$ is $\pi$-periodic,
\[
C_{-a}=\int_0^{\pi}\omega(a,t+\tfrac\pi2)\,dt
=\int_{\pi/2}^{3\pi/2}\omega(a,u)\,du=C_a.
\]
Thus \eqref{eq:Cbounds} holds for every $|a|<\tfrac12$. The identity \eqref{eq:beta-half} also holds there, its proof using only $\omega(a,\pi-t)=\omega(a,t)$.

The proof of Lemma \ref{lem:s1} now applies verbatim: $G_a(0)=-z_a(0)<0$ because $z_a(0)=(\tfrac12+a)^{1/2}>0$; by the previous paragraph $\beta(a,\tfrac\pi2)=\tfrac12C_a\in(\tfrac\pi2,\tfrac{\pi}{\sqrt2}]$, so $\cos\beta(a,\tfrac\pi2)<0$ and $G_a(\tfrac\pi2)>0$; and $\beta(a,\cdot)$ is
strictly increasing with $\beta(a,\tfrac\pi2)<\pi$, so $y_a>0$ on $(0,\tfrac\pi2]$ and \eqref{eq:Gprime} makes every zero in $(0,\tfrac\pi2]$ a simple upward crossing. Exactly one zero therefore occurs, and $s_1(a)<\tfrac\pi2$.

In \eqref{eq:fbc-clean} the left-hand side is now strictly \emph{negative}, because $a<0$, $s_1\in(0,\tfrac\pi2)$ and $\sin\beta(a,s_1)>0$; since $c(a)>0$ and $z_a>0$ this forces $\cos\beta(a,s_1)>0$, whence $\beta(a,s_1)\in(0,\tfrac\pi2)$,
$x_a(s_1)=\lambda_a(s_1)^{1/2}\cos\beta(a,s_1)>0$ and $R(a)<\tfrac\pi2$. In the expression for $\dot x_a$ both summands are again strictly negative at $s_1$: the first because $a<0$, $\sin(2s_1)>0$ and $\cos\beta(a,s_1)>0$, the second
because $\sin\beta(a,s_1)>0$. Thus $\dot x_a(s_1)<0$ and
Remark \ref{rem:side} applies unchanged.

Since $\beta(a,\cdot)$ is odd and strictly increasing, its range on $[-s_1,s_1]$ is an interval of length $2\beta(a,s_1)<\pi<2\pi$, so the injectivity argument of Proposition \ref{prop:embcont} goes through. For the
monotonicity of $x_a$ one modification is needed: here
$\beta(a,s)\in[0,\tfrac\pi2)$ for every $s\in[0,s_1]$, so $x_a>0$ on $[0,s_1]$ and only the case $u(s_c)>0$ of that proof can occur; it is excluded there by $\ddot u(s_c)=-2u(s_c)<0$. Hence $x_a$ is strictly decreasing on $[0,s_1]$, so $x_a\geq x_a(s_1)=\cos R(a)$ and $\Sigma_a\subset B(R(a))$.

The zero being simple, the analytic implicit function theorem applies at every $a\in(-\tfrac12,0)$ and the argument of Lemma \ref{lem:analytic} yields the real-analyticity of $s_1$ and of $R$ there. Finally $\beta(a,\cdot)$ is odd for
every $|a|<\tfrac12$, so $x_a,z_a$ are even and $y_a$ is odd, and the paragraph following \eqref{eq:Psi} applies verbatim: $\{\Psi_a\}$ is rotational and reflection-symmetric in the sense of \eqref{eq:reflection}. Since \eqref{eq:Aclosed} and Lemma \ref{lem:phinonzero} are computations in $c(a)$
and $z_a$ valid for $|a|<\tfrac12$, giving $A_a(\eta,\eta)^2=c(a)^2/z_a(s_1)^4>0$
and $|\varphi_a(0)|=1/\bigl(2c(a)\bigr)>0$, Theorem \ref{thm:exactnul} applies.
\end{proof}

\begin{remark}[A spectral reading of {\cite[Rem. 3.5]{dO24}}]\label{rem:dO35}
By Proposition \ref{prop:subhemi} the radius map sends $(-\tfrac12,0)$ into $(0,\tfrac\pi2)$, and \cite[Rem. 3.5]{dO24} asks, on numerical evidence, whether it is a bijection onto $(0,\tfrac\pi2)$. The dictionary gives that no member of this branch is parametrically degenerate if and only if $R'$ vanishes nowhere on $(-\tfrac12,0)$; conversely a critical point of $R$ there would produce a degenerate annulus \emph{below} the hemisphere, in the range of radii
to which \cite[Q. 6.6]{NZ1} refers. Two caveats. The nonvanishing of $R'$ is strictly stronger than injectivity, since $R$ may be injective near a critical point of odd order in the sense of \S\ref{sec:intro}; and the statement concerns the rotationally invariant, reflection-even sector only, so it leaves the full nullity open, for which see Problem \ref{prob:fullnullity}.
\end{remark}

\section{The hyperbolic side of the dichotomy, and open problems}\label{sec:remarks}

\subsection*{The Mori family in \texorpdfstring{$\Hyp^3$}{H3}}

The rotational free boundary minimal annuli $\Sigma_a\subset B(r(a))\subset\Hyp^3$, $a>\tfrac12$, of Mori \cite{Mori81} and do Carmo--Dajczer \cite{dCD83}, in the free boundary truncation of Li--Xiong \cite[\S 2]{LX18} and in the normalization of \cite{P1,P2}, form a reflection-symmetric rotational variable-radius family. Its two standing hypotheses are met: the nonvanishing of $A_a(\eta,\eta)$ holds by the closed form \eqref{eq:Aclosed-mori} below, in accordance with Lemma \ref{lem:bdrycurv}, and $\varphi_a\not\equiv 0$ holds for every $a>\tfrac12$ because \cite[Lem. 11.1]{P1} exhibits the constant nonzero Wronskian $W(\varphi_a,u_*)\equiv-a/K(a)$ of $\varphi_a$ against the axial boost field $u_*$, and a nonvanishing Wronskian forces both entries to be nontrivial. Consequently Theorem \ref{thm:exactnul} and Theorem \ref{thm:fold}(i)--(ii) apply to this family unconditionally.

We record the profile equations, both to fix notation and because they reduce Problem \ref{prob:mori} to two elementary monotonicity statements. In Fermi coordinates about the rotation axis the metric of $\Hyp^3$ reads $du^2+\cosh^2u\,dt^2+\sinh^2u\,d\theta^2$, where $u$ is the distance from the axis, $t$ the arclength along it and $\theta$ the rotation angle; the centre $p_0$ is the point $u=t=0$, and $\cosh\rho=\cosh u\cosh t$. A rotational surface is generated by a profile $s\mapsto(u(s),t(s))$, parametrised by arclength, its induced metric is $ds^2+b^2\,d\theta^2$ with $b:=\sinh u$ the profile radius, in the notation of the proof of Theorem \ref{thm:exactnul}, and its area is $2\pi\int b\,ds$. By the principle of symmetric criticality the surface is minimal precisely when the profile is critical for that functional; since $t$ is cyclic, this yields the first integral
\begin{equation}\label{eq:mori-clairaut}
b\,(1+b^2)\,\dot t=K,\qquad
\dot b^{\,2}=1+b^2-\frac{K^2}{b^2}=\frac{(b^2-B_0)(b^2+1+B_0)}{b^2},
\end{equation}
with $K>0$ constant along the profile and $B_0=\tfrac12\bigl(-1+\sqrt{1+4K^2}\bigr)$ the squared radius of the neck. The normalisation $K(a)=(a^2-\tfrac14)^{1/2}$ is precisely the one for which $B_0=a-\tfrac12$, so that the parameter range is exactly $a>\tfrac12$, the neck closing as $a\to(\tfrac12)^{+}$.

The second fundamental form of these annuli is available in closed form, exactly as in \eqref{eq:Aclosed}. By rotational symmetry $\partial_s,\partial_\theta$ are principal, so by minimality the principal curvatures are $\pm\mu_a$ with $\mu_a^2=A_a(\partial_s,\partial_s)^2$ and $|A_a|^2=2\mu_a^2$; the induced metric $ds^2+b^2d\theta^2$ has Gauss curvature $-\ddot b/b$, so the Gauss equation in $\Hyp^3$ gives $-\ddot b/b=-1-\mu_a^2$, that is $\mu_a^2=\ddot b/b-1$. Differentiating the second identity in \eqref{eq:mori-clairaut} gives $2\dot b\,\ddot b=2\dot b\,\bigl(b+K^2b^{-3}\bigr)$, and $\dot b$ vanishes only at the neck $b^2=B_0$, so that $\ddot b=b+K^2b^{-3}$ along the whole profile by continuity. Therefore
\begin{equation}\label{eq:Aclosed-mori}
A_a(\partial_s,\partial_s)^2=\frac{K(a)^2}{b^4},\qquad |A_a|^2=\frac{2\,K(a)^2}{b^4}\qquad\text{on all of }\Sigma_a .
\end{equation}
In particular $|A_a|$ never vanishes, and since $\eta=\pm\partial_s$ along the boundary, $A_a(\eta,\eta)^2=K(a)^2/b_0(a)^4>0$; being continuous and nowhere zero on the connected surface $\Sigma_a$, the function $A_a(\partial_s,\partial_s)$ has a constant strict sign, the same on both boundary circles. Thus for the Mori family too the conclusions of Lemma \ref{lem:bdrycurv} hold by direct computation, independently of \cite{NZ1}.

Let $s_0(a)$ be the free boundary parameter and write $b_0(a):=b(s_0(a))$. Orthogonality of the profile to the level sets of $\rho$ reads $\dot b\sinh t=K\cosh t$ at $s_0$, that is $\tanh t(s_0)=K/\dot b(s_0)$; substituting into $\sinh^2\rho=(1+b^2)\cosh^2t-1$ and using $\dot b^{\,2}-K^2=(1+b^2)(b^2-K^2)/b^2$, which follows from \eqref{eq:mori-clairaut}, one obtains the closed relation between the radius and the profile,
\begin{equation}\label{eq:mori-radius}
\sinh^2 r(a)=\frac{b_0(a)^4}{b_0(a)^2-K(a)^2} ,
\end{equation}
which is \cite[Prop. 8.10]{P1}. Since $\partial_B\log\bigl(B^2/(B-K^2)\bigr)=(B-2K^2)/\bigl(B(B-K^2)\bigr)$, the strict pinching condition $b_0^2>2K^2$ is exactly the statement that the radius increases with the profile radius at the boundary. Differentiating \eqref{eq:mori-radius} with respect to $K$, a strictly increasing function of $a$, gives
\begin{equation}\label{eq:mori-reduction}
\frac{d}{dK}\log\sinh^2 r=\frac{b_0^2-2K^2}{b_0^2\,\bigl(b_0^2-K^2\bigr)}\cdot\frac{d\,b_0^2}{dK}+\frac{2K}{b_0^2-K^2},
\end{equation}
whose second summand is positive. Hence the strict pinching condition, together with the strict increase of $K\mapsto b_0^2$, forces $r$ to be strictly increasing, and therefore, by Theorem \ref{thm:dictionary}, forces the Mori family to be free of parametric degenerations: Problem \ref{prob:mori} reduces to those two monotonicity statements.

The fine analysis of this family is carried out in \cite{P1}, to which we refer; in particular \cite[Thm. 8.17]{P1} establishes the strict pinching condition on $(\tfrac12,1]$, its validity on $(1,\infty)$ being left open there apart from the limiting regimes, and \cite[Thm. 11.5 and Thm. 11.6]{P1} establish the equivalence, under it, between the positivity of the parametric field on the profile and the \emph{strict increase} of $a\mapsto\sinh r(a)/K(a)$. We emphasise that it is the strict increase, and not mere strict monotonicity, that is relevant: since $K(a)K'(a)=a$, one has
\begin{equation}\label{eq:H-log}
K^2\,\frac{\bigl(\sinh r/K\bigr)'}{\sinh r/K}=K^2r'\coth r-a ,
\end{equation}
so $r'(a)=0$ forces the left-hand side to be negative, and a strictly decreasing $\sinh r/K$ is therefore compatible with a critical point of the radius, whereas a strictly increasing one is not.

Identity \eqref{eq:H-log} also converts the local analytic result of \cite{P1} into the hyperbolic half of the dichotomy announced in the introduction.

\begin{corollary}[No parametric degeneration in the Mori family near the closing neck]\label{cor:mori-local}
By \cite[Cor. 11.10]{P1}, the function $H(a)=\sinh r(a)/K(a)$ is strictly increasing on a right neighbourhood of $\tfrac12$. Consequently there exists $\delta_0>0$ such that $r'(a)>0$ for every $a\in(\tfrac12,\tfrac12+\delta_0)$; on that range the Mori family carries no parametric degeneration, $\varphi_a\notin\Ker(\Sigma_a,r(a))$, and
\[
\dim\Ker_0^{\mathrm{ev}}(\Sigma_a)=0\qquad\text{for every }a\in\bigl(\tfrac12,\tfrac12+\delta_0\bigr).
\]
\end{corollary}

\begin{proof}
By \cite[Cor. 11.10]{P1} there is $\delta_0>0$ such that $H'(a)>0$ on $(\tfrac12,\tfrac12+\delta_0)$, where $H(a)=\sinh r(a)/K(a)$. Since $H>0$, identity \eqref{eq:H-log} gives $K^2r'\coth r-a=K^2H'/H>0$ there, whence $K^2r'\coth r>a>0$ and therefore $r'(a)>0$. The two standing hypotheses of Theorem \ref{thm:exactnul} hold for this family, as recalled at the beginning of this subsection. Theorem \ref{thm:dictionary} then gives $\varphi_a\notin\Ker(\Sigma_a,r(a))$, and Theorem \ref{thm:exactnul} gives the stated vanishing.
\end{proof}

In view of Corollary \ref{cor:monotone}, and of the two standing hypotheses being met, the absence of rotationally invariant even Jacobi--Robin fields along the Mori family is equivalent to $r'\neq 0$, hence to the local injectivity of the radius map. Corollary \ref{cor:mori-local} settles this near the closing neck. The following therefore remains open.

\begin{problem}\label{prob:mori}
Is the radius function $a\mapsto r(a)$ of the Mori family strictly monotone on $(\tfrac12,\infty)$ --- equivalently, since $r(a)\to 0$ as $a\to(\tfrac12)^+$ and $r(a)\to\infty$ as $a\to\infty$ by \cite{P2}, a bijection onto $(0,\infty)$? Equivalently, by Theorem \ref{thm:dictionary}, is the family free of parametric degenerations? The question is raised in \cite[Rem.~4.3]{dO-thesis}, where numerical evidence for strict monotonicity is given (Fig.~4.2); as noted there, the resulting uniqueness is asserted without justification in \cite{LX18} and cited from there in \cite{Med23}. It is raised independently in \cite[Question 6.3]{P2}; by Corollary \ref{cor:mori-local} the answer is affirmative on a right neighbourhood of $\tfrac12$, and by \eqref{eq:mori-reduction} it suffices, on the remaining range, to establish the strict pinching condition $b_0^2>2K^2$ together with the strict increase of $K\mapsto b_0^2$.
\end{problem}

\subsection*{Problems in spherical caps}

\begin{problem}\label{prob:caps}
Determine the exact number and the nondegeneracy of the critical points of the radius map $a\mapsto R(a)$ on $(0,\tfrac12)$. The numerical discussion in \cite[Rem. 1.3]{NZ1} exhibits a single interior maximum of the radius, but makes no assertion about its nondegeneracy or about the absence of further critical points. Uniqueness of the maximizer would make the annulus of Definition \ref{def:degannulus} the unique degenerate member of the family. As for nondegeneracy, we prove $R'(a_*)=0$ but make no claim about $R''(a_*)$, and this costs less than it may appear: by Theorem \ref{thm:main}(iv) the vanishing of $\dim\Ker_0^{\mathrm{ev}}$ on a punctured neighborhood of $a_*$ is unconditional, real-analyticity alone forcing the zeros of $R'$ to be isolated. Only the local two-to-one description near $a_*$, that is, the statement that each radius slightly below $R(a_*)$ is attained by exactly two parameters near $a_*$, would require $R''(a_*)\neq 0$. Proposition \ref{prop:Rprime} reduces both questions to the elementary relation \eqref{eq:critical}: the critical points of $R$ are exactly the zeros of
\[
a\longmapsto \tfrac12\sin\bigl(2s_1(a)\bigr)+2a\Bigl(z_a(s_1)^2\cos\bigl(2s_1(a)\bigr)+a\sin^2\bigl(2s_1(a)\bigr)\Bigr)s_1'(a),
\]
an expression in which the only transcendental ingredient is $s_1$, through \eqref{eq:s1prime}. Both the uniqueness of the zero and its nondegeneracy are therefore questions about the single scalar function $s_1$ and its first two derivatives, rather than about the radius map itself.
\end{problem}

\begin{problem}\label{prob:fullnullity}
Compute the full nullity $\nul(\Sigma_{a_*})$ of a degenerate annulus. Theorem \ref{thm:main}(iii) exhibits one kernel element in the reflection-even part of mode zero, outside the Killing-induced ones, and Proposition \ref{prop:killnul} contributes two Killing-induced elements in the modes $|k|=1$, whence $\nul(\Sigma_{a_*})\geq 3$. Three sectors remain unaccounted for, and the lower bound is not claimed to be sharp: the reflection-odd part of mode zero, on which Theorem \ref{thm:exactnul} says nothing; the possible excess of the modes $|k|=1$ over the two Killing-induced elements, Proposition \ref{prop:killnul} being a lower bound; and the modes $|k|\geq 2$. We note that at every \emph{noncritical} parameter Corollary \ref{cor:signchange} gives, in any Fourier mode, that every nonzero Jacobi--Robin field of $\Sigma_a$ changes sign on $\partial\Sigma_a$, which constrains the corresponding analysis.
\end{problem}

\begin{problem}\label{prob:otherbranches}
De Oliveira \cite[Thm. 3.2]{dO24} produces, for suitable $a$, countably many nested free boundary minimal annuli $X_a:[-s_i(a),s_i(a)]\times\Sph^1\to\Sph^3$ with $i\geq 1$. Do the higher branches $a\mapsto R_i(a)$, $i\geq 2$, also fold, and do they produce further degenerate annuli?
\end{problem}

\subsection*{Higher dimensions and non-rotational families}

Theorem \ref{thm:defect} holds in every dimension and requires no symmetry. The rotational hypersurface families of do Carmo--Dajczer \cite{dCD83} in $\Hyp^{n+1}$ and $\Sph^{n+1}$, truncated at their free boundary, provide variable-radius families for which index, nullity and radius monotonicity are, to the author's knowledge, entirely open. Likewise, the dictionary of Theorem \ref{thm:dictionary} applies to the non-rotational one-parameter families of free boundary minimal annuli in geodesic balls of $\Hyp^3$ constructed by Cerezo \cite{Cer25}, and to the ambient families of \cite{CFM25,FHM23}: for each such family the parametric degenerations are exactly the critical points of the corresponding radius function.

\section*{Declarations}

\subsection*{Conflict of interest}

The author declares no conflict of interest.

\subsection*{Data availability}

Data sharing is not applicable to this article as no datasets were generated or analysed during the current study.

\end{document}